\documentclass[sn-mathphys-num]{sn-jnl}
\usepackage[utf8]{inputenc}
 \usepackage{a4wide}
\usepackage{amssymb}
\usepackage{amsmath}
\usepackage{amsthm}
\usepackage{hyperref}
\usepackage{breqn}
\usepackage{xcolor}
\usepackage{mathtools}
\usepackage[numbers]{natbib}
\usepackage{bm}
\usepackage{enumitem}
\usepackage{mathrsfs}
\usepackage{graphicx}
\usepackage{url}
\usepackage{placeins}
\usepackage[english]{babel}
\usepackage{array}
\usepackage{booktabs}
\usepackage{multirow}
\numberwithin{equation}{section}

\usepackage{graphicx}
\newtheorem{lemma}{Lemma}[section]
\theoremstyle{remark}

\def \R{{\mathbb R}}

\def \x{{\bm x}}
\def \b{{\bf b}}

\def\bmatrix#1{\left[\begin{matrix}
		#1
	\end{matrix}\right]}

\def \diag{\mathrm{diag}}

\def  \0{\bf 0}

\def \x{\bm x}

\def\bmatrix#1{\left[ \begin{matrix} #1 \end{matrix} \right]}

\def \R{{\mathbb R}}

\theoremstyle{thmstyleone}%
\newtheorem{theorem}{Theorem}
\def \x{\bm{x}}

\def \v{\bm{v}}
\def \p{\bm{p}}

\def \0{{\bf 0}}

\usepackage{algorithm}
\usepackage{euscript}
\usepackage{mathrsfs}
\usepackage{algpseudocode}

 \theoremstyle{definition}%
\newtheorem{example}{Example}%
\newtheorem{remark}{Remark}%

\begin{document}

\title[On a Class of Block-Regularized Preconditioned Iterative Methods for Indefinite Least-Squares Problems]{On a Class of Block-Regularized Preconditioned Iterative Methods for Indefinite Least-Squares Problems}



\author[1]{\fnm{Pinki} \sur{Khatun}}\email{pinkikhatun.nla@gmail.com}

\affil[1]{\orgdiv{Department of Industrial Engineering}, \orgname{University of Florence}, \orgaddress{\street{Viale Morgagni 40/44}, \city{Florence}, \postcode{50134},  \country{Italy}}}

\abstract{This paper proposes a novel block-regularized splitting (BRS) framework for the efficient solution of indefinite least-squares (ILS) problems. Based on a block-wise regularization strategy incorporated into a matrix splitting scheme, we develop a BRS iterative method together with an effective BRS preconditioner. The convergence of the proposed iterative method is rigorously analyzed, and spectral bounds for the BRS-preconditioned matrix are established. To further enhance computational performance, we introduce a relaxed BRS (RBRS) preconditioner, which provides improved spectral properties and significantly accelerates the convergence of Krylov subspace methods. Extensive numerical experiments on both dense and sparse test problems demonstrate that the proposed BRS and RBRS preconditioners consistently outperform existing approaches in terms of iteration count, computational time, and overall efficiency. These results highlight the effectiveness and robustness of the proposed block-regularized splitting framework for solving large-scale ILS problems.
	}





\maketitle

\section{Introduction}
The indefinite least-squares (ILS) problem, first introduced by \citet{ILS1998}, is an extension of the classical linear least-squares problem. It can be formally expressed as:
\begin{align}\label{ILS1}
  \text{ILS}: ~~ \min (\bm{b}-A\x)^T\mathcal{J}(\bm{b}-A\x),
\end{align}
where $A \in \mathbb{R}^{m \times n}$ be a large sparse matrix with $m \geq n$, 
$\bm{b} \in \mathbb{R}^m$, and 
\[
\mathcal{J} = \bmatrix{I_p& \bf 0\\ \bf 0 & -I_q}
\]
denotes the signature matrix with $p+q = m$. 
The ILS problem reduces to the standard least-squares (LS) problem when either $p = 0$ or $q = 0$, 
in which case the corresponding quadratic form is definite. The ILS problem has found wide applications in areas such as total least-squares \cite{TLS1980, TLS1991} and $H^\infty$ smoothing \cite{Hinfinity2, Hinfinity}.
In contrast, when $p\,q > 0$, the quadratic form becomes indefinite. 
Thus, problem~\eqref{ILS1} involves minimizing an indefinite quadratic form associated with the signature matrix $\mathcal{J}$. 
Since $\mathcal{J}$ has both positive and negative inertia, the resulting minimization problem is not necessarily solvable.
The associated normal equation is
\begin{align}\label{ILS2}
    A^T\mathcal{J}A\x=A^T\bm{b}.
\end{align}
A critical structural property of the ILS problem \eqref{ILS1} resides in its Hessian matrix formulation, which is $A^T\mathcal{J}A$. Therefore, the ILS problem admits a unique solution if and only if 
\begin{center}
$A^T \mathcal{J} A$ is symmetric positive definite (SPD),     
\end{center}
which we assume to hold throughout this paper. Note that this condition also implies that $p\geq n$ and  $A(1:p, 1:n)$ has full column rank.

 Over the years, extensive research has been devoted to the computation, perturbation analysis, and applications of this concept. In this paper, we mainly focus on numerical methods.  
For small and dense ILS problems, \citet{ILS1998} developed a stable direct approach known as the QR-Cholesky method. This method first performs a QR factorization of $A$ ($A = QR$, with $Q^T Q = I$ and $R$ an upper triangular matrix), then solves 
\begin{equation}
(Q^T \mathcal{J} Q) y = Q^T \mathcal{J} b
\end{equation}
via Cholesky factorization, and finally computes the solution as $\bm{x} = R^{-1} y$. To improve efficiency, a method based on hyperbolic QR factorization was later introduced in \cite{ILAS2003}, which requires fewer operations than the QR-Cholesky method. \citet{Xu2004ILS} further enhanced stability by applying hyperbolic QR factorization to the normalized matrix, thereby ensuring backward stability. A unified analysis of these three methods was subsequently presented in \cite{Unified2021}.  

For large and sparse ILS problems, direct methods are computationally impractical, making iterative methods essential. Formulating the indefinite least-squares problem as a double saddle-point system exposes its underlying block structure, enabling the use of stable Krylov methods and effective preconditioning strategies that are not readily available in the original formulation.
\citet{BSOR2014ILS} proposed a block SOR method with a relaxation parameter, which was recently refined by \citet{USSOR2020ILS} through the USSOR method, incorporating two parameters for improved performance. Further, \citet{Uzawa2025ILS} introduced the variable parameter Uzawa method. Although stationary iterative methods for solving the ILS problem often exhibit slow convergence, their efficiency can be significantly enhanced by incorporating suitable preconditioners within Krylov subspace methods. The use of a preconditioner transforms the original system into an equivalent one, to which a Krylov subspace method is then applied. An effective preconditioner should satisfy two key properties: $(1)$ its inverse should be computationally inexpensive to apply, and $(2)$ the resulting preconditioned system should have a substantially reduced condition number. Nevertheless, research on preconditioning techniques for the ILS problem remains relatively limited. Early contributions include preconditioned conjugate gradient methods by \citet{PCG2011ILS} and preconditioners based on the incomplete hyperbolic Gram–Schmidt process \cite{Gram2013ILS}. More recently, \cite{BSP2025OILS} introduced three block-splitting preconditioners. Further studies for preconditioning techniques have been done in \cite{ahmad2025robust, ahmad2026class,li2026double,  li2026gmres, BUT2025,  AAU2025, ADI2025, SS2026, li2026two_nys}.

Modern large-scale data assimilation demands highly stable optimization frameworks capable of processing sparse, indefinite measurement matrices without sacrificing numerical precision.   This paper addresses these challenges by developing an efficient preconditioning framework for the augmented system formulation of the ILS problem, which preserves matrix sparsity, avoids squaring the condition number, and ensures stable and scalable convergence. By transforming the minimization problem into an augmented system and merging the idea of the shift-splitting technique and the Schur preconditioner technique, we propose a new regularized splitting of the coefficient matrix. This splitting gives a stationary iterative method and a preconditioner to accelerate the rate of the Krylov subspace iterative methods.

The main contribution of the paper is highlighted as follows:
\begin{itemize}
    \item We introduce a block‐regularized splitting of the coefficient matrix and develop a corresponding stationary iterative method for solving the ILS problem.

\item Building on this splitting, we design a block‐regularized splitting (BRS) preconditioner to accelerate the convergence of Krylov subspace methods. In addition, we present a relaxed variant of the BRS preconditioner, termed as the RBRS preconditioner, for further performance improvement.
\item  The convergence analysis of the proposed iterative method is performed in detail, which indicates the unconditional convergence of the method.
\item We conduct extensive numerical experiments on two sets of test problems to evaluate the performance of the proposed preconditioners. The results demonstrate that the BRS and RBRS preconditioners consistently outperform existing block-splitting preconditioners in terms of iteration counts, computational time, and robustness.
\end{itemize}

The remainder of the paper is organized as follows. In Section~\ref{sec2}, the ILS problem is reformulated in an augmented form, and a novel matrix splitting of the resulting coefficient matrix is introduced. Section~\ref{SEC3} is devoted to the convergence analysis of the proposed iterative method. In Section~\ref{SEC4}, we present the implementation details and conduct a spectral analysis of the proposed preconditioner. Numerical experiments illustrating the performance of the method are reported in Section~\ref{SEC5}. Finally, concluding remarks and directions for future work are given in Section~\ref{SEC6}.

\vspace{2mm}\textbf{Notation:} Throughout this work, $\R^{m\times n}$
 represents the set of all real matrices with $m$ rows and $n$ columns. The symbol $I_m$ denotes the identity matrix of order $m.$ We use $\0$ to denote the zero matrix of appropriate size. We use $\rho(A)$ to denote the spectral radius of the matrix $A.$ The notation $\diag(X_1,~X_2)$ denotes the block diagonal matrix with diagonal blocks $X_1$ and $X_2.$ We use $\text{Null}(A)$ to denote the null space of the matrix $A.$

 \section{Block Regularized Splitting  of the Coefficient Matrix of the Augmented System}\label{sec2}

To develop the iteration method and preconditioner, we first partition the matrix $A$ and $\bm{b}$ in the ILS problem \eqref{ILS1} as follows:
\begin{align}
    A=\bmatrix{A_1\\ A_2},~~~ \bm{b}=\bmatrix{b_1\\ b_2},
\end{align}
where $A_1\in \R^{p\times n}$ is full column rank, $A_2\in \R^{q\times n},$ $b_1\in \R^{p}$ and $b_2\in \R^{q}.$ Using this partitioning, the normal equation in \eqref{ILS2} can be expresse as:
\begin{align}\label{ILS3}
    \bmatrix{A_1 & \bf 0& I_p\\ A_2 & I_q & \bf 0\\ \bf 0&-A_2^T &A_1^T}\bmatrix{\x \\ \p\\ \v}=\bmatrix{b_1\\b_2\\ \bf 0},
\end{align}
where $\delta= [\v^T, \p^T]^T=\bm{b}-A\x.$
The above system can be further equivalently written as:
\begin{align}\label{ILS4}
    \mathcal{K}\bm{u}:=\bmatrix{I_p& A_1 & \bf 0\\ -A_1^T & \bf 0 & A_2^T\\ \bf 0&A_2 & I_q}\bmatrix{ \v\\ \x \\ \p}=\bmatrix{b_1\\ \bf 0 \\b_2}.
\end{align}

Consider the following block regularized matrix splitting:

\begin{align}\label{BRS:spllitng}
    \mathcal{K}
    = \bmatrix{(\alpha+1)I_p & A_1& \bf 0\\ -A_1^T &S &\bf 0\\ \bf 0 & A_2 & (\alpha+1)I_q}- \bmatrix{\alpha I_p & \bf 0& \bf 0\\ \bf 0 &S &-A_2^T\\ \bf 0 & \bf 0 & \alpha I_q}:=\mathscr{P}_{\text{BRS}}-\mathcal{Q}_{\text{BRS}},
\end{align}
where $S\in \R^{n\times n}$ is a symmetric positive definite (SPD) matrix and $\alpha>0.$ 
Using the above matrix splitting, we have proposed the following iteration scheme for solving the ILS problem \eqref{ILS1}.

\vspace{2mm}
\noindent\textbf{Block regularized splitting (BRS) iteration method:} Assume that the solution of the ILS problem \eqref{ILS1} exists and satisfies uniqueness assumptions. Given an initial guess $\bm{u}_0=[\v^T, \bm{x}^T, \p^T]^T,$ until iteration sequence $ \bm{u}_k=[\v^T, \bm{x}^T, \p^T]^T$ converges, compute:
\begin{align}
    \bm{u}_{k+1}=\mathscr{P}_{\text{BRS}}^{-1}\mathcal{Q}_{\text{BRS}}\bm{u}_k+\bm{d}.
\end{align}

The algorithmic version of the BRS Iterative scheme to solve ILS problem~\eqref{ILS1} is given in following: 



\begin{algorithm}
\caption{The BRS Iterative scheme for solving the ILS problem}
\label{alg:ILS}

\begin{algorithmic}[1]
\Require The ILS problem~\eqref{ILS1} admits a unique solution.
\State \textbf{Initialization:} Choose an initial guess
\[
\bm{u}_0 = \bmatrix{ \v_0^T, & \bm{x}_0^T, & \p_0^T}^T.
\]

\State \textbf{Iteration:} For $k = 0,1,2,\ldots$, update
\[
\bm{u}_{k+1} = \mathcal{G} \bm{u}_k + \bm{d},
\]
where $\bm{u}_k = \bmatrix{ \v_k^T, & \bm{x}_k^T, & \p_k^T}^T,$ $\mathcal{G}=\mathscr{P}_{\text{BRS}}^{-1}\mathcal{Q}_{\text{BRS}}$ and $\bm{d}=\mathscr{P}_{\text{BRS}}^{-1}\bm{b}.$
\State \textbf{Stopping criterion:} 
Terminate when
\[
\|\bm{u}_{k+1} - \bm{u}_k\| < \epsilon,
\]
for a prescribed tolerance $\epsilon > 0$.
\end{algorithmic}
\end{algorithm}
\noindent

The matrix $\mathcal{G}$ in Algorithm \ref{alg:ILS} is the iteration matrix of the BRS iteration method. Therefore, according to \cite{YSAAD}, the BRS iterative scheme converges to the unique solution of the ILS problem if and only if the spectral radius of the iteration matrix $\mathcal{G}$ satisfies
\[
\rho(\mathcal{G})=\rho(\mathscr{P}^{-1}_{\text{BRS}}\mathcal{Q}_{\text{BRS}}) < 1.
\]
This condition guarantees that the error between successive iterates decays geometrically, ensuring both stability and convergence of the iterative method.

Moreover, the BRS splitting in \eqref{BRS:spllitng} introduces a new preconditioner to accelerate the convergence of Krylov subspace methods, which is  given by:
\begin{align}
    \mathscr{P}_{\text{BRS}}=\bmatrix{(\alpha +1)I & A_1& \bf 0\\ -A_1^T&S &\0\\ \bf 0 & A_2 & (\alpha+1)I}.
\end{align}
We termed this preconditioner the BRS preconditioner.


\section{Convergence Analysis of the Proposed BRS Iterative Method}\label{SEC3}
     This section is devoted to the convergence analysis of the BRS iterative method. Any stationary iteration method discussed in Algorithm \ref{alg:ILS} will converge if and only if its iteration matrix has a spectral radius less than $1$ \cite{YSAAD}. Thus, we need to show that $\rho(\mathcal{G})<1.$

     The following result is important to prove the main theorem.
     \begin{lemma}[\cite{BOOK1971}]\label{Lemma1}
Both roots of the real quadratic equation 
\[
x^2 - \beta x + \alpha = 0
\]
lie inside the unit circle (i.e., have modulus less than $1$) if and only if 
\[
|\alpha| < 1 \quad \text{and} \quad |\beta| < 1 + \alpha.
\]
\end{lemma}

     Let $\lambda$ be an eigenvalue of the iteration matrix $\mathcal{G}$ and $\bm{p}=[u^T,v^T,w^T]^T\in \R^{p+n+q}$ is the corresponding eigenvector Then
     \begin{align}
         \mathcal{Q}_{\text{BRS}}\bmatrix{u\\v\\w}=\lambda \mathscr{P}_{\text{BRS}}\bmatrix{u\\v\\w}
     \end{align}
     The above is equivalent to the following  linear systems of equations:
     \begin{align}\label{eq32}
       &  \lambda (\alpha +1) u+\lambda A_1v=\alpha u,\\ \label{eq33}
         &\lambda Sv-\lambda A_1^Tu= Sv-A_2^Tw,\\ \label{eq34}
         &\alpha w=\lambda A_2v+\lambda (\alpha+1)w.
     \end{align}
     We have the following result.
\begin{lemma}
    Assume $A_1 \in \mathbb{R}^{p \times n}$ is of full column rank, 
$A_2 \in \mathbb{R}^{q \times n}$, with $p+q = m$, and let 
$b_1 \in \mathbb{R}^p$, $b_2 \in \mathbb{R}^q$. 
Let $\alpha > 0$ be a positive parameter and $S$ be an SPD matrix. Let   $\lambda$ be an eigenvalue of the iteration matrix $\mathcal{G},$  then
\(
\lambda \neq \pm 1.
\)
\end{lemma}
  \begin{proof}
      Let $\bm{p}=[u^T,v^T,w^T]^T\in \R^{p+n+q}$ be an eigenvector of $\mathcal{G}$ corresponding to the eigenvalue $\lambda.$ If $\lambda=1,$ then from \eqref{eq32}-\eqref{eq34}, we obtain:
      \begin{align*}
          u+A_1v=\0, ~~ A_1^Tu-A_2^Tw=\0,~~ A_2v+w=\0.
      \end{align*}
      The above equation together are equivalent to $\mathcal{K}\bm{p}=\0,$ implies that $\bm{p}=\0.$ This is a contradiction with the fact that, $\bm{p}$ is an eigenvector. Thus, $\lambda\neq 1.$ Next, consider $\lambda=-1.$ Then, from \eqref{eq32}-\eqref{eq34}, we have 
      \begin{align*}
          u+A_1v=\0, ~~ 2Sv-A_1^Tu-A_2^Tw=\0,~~ A_2v+(2\alpha+ 1)w=\0.
      \end{align*}
      Thus, substituting $u=-\frac{1}{2\alpha+1}A_1^Tv$ and $w=-\frac{1}{2\alpha+1}A_2w$ in $2Sv-A_1^Tu-A_2^Tw=\0,$ we obtain:
      \begin{align*}
          \left(\frac{1}{(2\alpha+1)}A_1^TA_1+\frac{1}{(2\alpha+1)}A_2^TA_2+2S\right)v=\0.
      \end{align*}
      Since the matrix $\left(\frac{1}{(2\alpha+1)}A_1^TA_1+\frac{1}{(2\alpha+1)}A_2^TA_2+2S\right)$ is SPD, we have $v=\0.$ This gives $u=\0$ and $w=\0,$ which is not possible as $\bm{p}$ is an eigenvector. Thus, $\lambda\neq -1.$ Hence, the proof follows.
  \end{proof}

In the following theorem, we derive the unconditional convergence properties of the BRS iterative method.
  \begin{theorem}
      Assume $A_1 \in \mathbb{R}^{p \times n}$ is of full column rank and 
$A_2 \in \mathbb{R}^{q \times n}$, where $p+q = m$, with 
$b_1 \in \mathbb{R}^p$ and $b_2 \in \mathbb{R}^q$. 
Then, for any initial guess vector the BRS iterative method is unconditionally convergent 
for any $\alpha > 0$ and any SPD matrix $S$, 
that is,
\[
\rho(\mathcal{G}) < 1.
\]
  \end{theorem}
  \begin{proof}
     Let $\bm{p}=[u^T,v^T,w^T]^T\in \R^{p+n+q}$ be an eigenvector of $\mathcal{G}$ corresponding to the eigenvalue $\lambda.$ Then \eqref{eq32}-\eqref{eq34} holds. Consider $v=0.$ Then, from \eqref{eq32}, we have 
      \begin{align}
          &\lambda (\alpha +1)u- \alpha u=\0.
      \end{align}
      If $\lambda \neq \frac{\alpha}{\alpha+1},$ from above we have $u=\bf 0$ and $w=\bf 0,$ which is a contradiction to the fact that $\bm{p}$ is an eigenvector. Hence, $\lambda = \frac{\alpha}{\alpha+1}.$ Then $|\lambda| < 1$ if and only if 
\(
\left| \frac{\alpha}{\alpha + 1} \right| < 1,
\)
which holds true for all $\alpha > 0$. Similarly, if $v\neq \0$, then $\lambda \neq \frac{\alpha}{\alpha+1}.$ From, \eqref{eq32} and \eqref{eq34}, we have
\begin{align}
    u=\frac{\lambda A_1v}{\alpha-\lambda (\alpha +1)}~\text{and}~ w=\frac{\lambda A_2v}{\alpha-\lambda (\alpha +1)},
\end{align}
respectively. Substituting these values in \eqref{eq33}, we obtain:
\begin{align}\label{eq37}
    \lambda Sv-\frac{\lambda^2 A_1^TA_1v}{\alpha-\lambda (\alpha +1)}=-\frac{\lambda A_2^TA_2v}{\alpha-\lambda (\alpha +1)}+Sv.
\end{align}
Multiplying by $\frac{v^T}{v^Tv}$ from the left on both sides of \eqref{eq37}, we obtain:
\begin{align}\label{eq38}
     (\lambda -1)\frac{v^TSv}{v^Tv}-\frac{\lambda^2 }{\alpha-\lambda (\alpha +1)}\frac{v^TA_1^TA_1v}{v^Tv}=-\frac{\lambda }{\alpha-\lambda (\alpha +1)}\frac{v^TA_2^TA_2v}{v^Tv}.
\end{align}
Define \begin{align*}
 s=\frac{v^TSv}{v^Tv},~~   p=\frac{v^TA_1^TA_1v}{v^Tv},~~ q=\frac{v^TA_2^TA_2v}{v^Tv}.
\end{align*}
Then $s>0$ as $S$ is SPD and $p,q\geq 0.$ Therefore, \eqref{eq38} becomes
\begin{align}\label{eq39}
    \lambda^2-\frac{2\alpha s+s+q}{(\alpha+1)s+p}\lambda +\frac{\alpha s}{(\alpha+1)s+p}=0.
\end{align}
By Lemma \ref{Lemma1}, the roots of the real quadratic equation \eqref{eq39} satisfy $|\lambda|<1$ if and only if 
\begin{align*}
    \left|\frac{\alpha s}{(\alpha+1)s+p}\right|<1 ~ ~\text{and}~~ \left|\frac{2\alpha s +s+q}{(\alpha+1)s+p}\right|<1+\frac{\alpha s}{(\alpha+1)s+p}.
\end{align*}
Now the inequality holds as $s>0$ and $p\geq 0.$ For the second inequality, we get:
\begin{align*}
    2\alpha s +s+q< (\alpha+1)s+p+\alpha s~~\text{and}~~ -(\alpha+1)s-p-\alpha s<  2\alpha s +s+q.
\end{align*}
The above holds as $p-q>0$ as $A^T\mathcal{J}A$ is SPD and $s>0,$ $p, q\geq 0.$ Hence, this completes the proof.
  \end{proof}
\section{Proposed BRS and RBRS Preconditioners}\label{SEC4}
In this section,  we have introduced a relaxed version of the BRS preconditioner
by omitting the $\alpha I$ term from the (1,1) and (3,3) block from the $\mathscr{P}_{\text{BRS}}$  as follows: 
\begin{align}
    \mathscr{P}_{\tt{RBRS}}=\bmatrix{I & A_1& \bf 0\\ -A_1^T&S &\0\\ \bf 0 & A_2 & I}.
\end{align}
Further, this section outlines the algorithmic procedures for implementing the BRS and RBRS preconditioners. Additionally, we analyze the eigenvalue distributions of the corresponding preconditioned matrices.

\subsection{ Algorithmic Implementation of the BRS Preconditioner}
In this subsection, we present the algorithmic implementation of $\mathscr{P}_{\text{BRS
     }}$, which incorporates within Krylov subspace methods such as  generalized minimum residual (GMRES) \cite{saad1986gmres}. At each iteration of the BRS iteration method, or when applying the BRS preconditioner within a Krylov subspace method, it is necessary to solve the generalized residual equation of the following form:
     \begin{align}
         \mathscr{P}_{\text{BRS}}\bm{w}=\bm{r},
     \end{align}
     where $\bm{w}=[w_1^T,w_2^T,w_3^T]^T$ and $\bm{r}=[r_1^T,r_2^T,r_3^T]^T.$
     The preconditioner $\mathscr{P}_{BRS
     }$ can be decomposed as:
     \begin{align}
    \mathscr{P}_{\text{BRS}}=\bmatrix{I & \bf 0& \bf 0\\ -\frac{1}{\alpha +1} A_1^T&I &\bf 0\\ \bf 0 &A_2 X^{-1} & I} \bmatrix{(\alpha +1)I & A_1& \bf 0\\ \bf 0 &X &\0\\ \bf 0 & \bf 0 & (\alpha+1)I},
\end{align}
where $X=S+\frac{1}{\alpha +1}A_1^TA_1.$
			
\begin{algorithm}
			\caption{Computation of  $\mathscr{P}_{\text{BRS}}\bm{w}=\bm{r}$}
                \label{alog2}
			\begin{algorithmic}[1]
				\State 
                \textbf{Input:} A positive real number $\alpha,$ $A=\bmatrix{A_1\\ A_2}\in \R^{m\times n},$ a SPD matrix $\mathcal{S}\in \R^{n\times n}$ and $\bm{r}=[r_1^T,r_2^T,r_3^T]^T.$
                \State \textbf{Output:} $\bm{w}=[w_1^T,w_2^T,w_3^T]^T.$
                \State \textbf{Steps:}
                \State  Solve $Xz_1=r_2$ for $z_1,$ where $X=S+\frac{1}{\alpha +1}A_1^TA_1$;
                \State  Solve $Xz_2=\frac{1}{\alpha +1}A_1^Tr_1$ for $z_2$;
                \State  Compute  $w_3=\frac{1}{\alpha+1}(r_3-A_2(z_1 +z_2))$:
                \State  Solve $Xw_2=r_2+\frac{1}{\alpha +1}A_1^Tr_1$ for $w_2;$
                \State  Compute $w_1=\frac{1}{\alpha + 1}(r_1-A_1w_2).$ 
			
   \end{algorithmic}
		\end{algorithm}


\begin{algorithm}
			\caption{Computation of  $\mathscr{P}_{R BRS}\bm{w}=\bm{r}$}
                \label{alog3}
			\begin{algorithmic}[1]
				\State 
                \textbf{Input:} A positive real number $\alpha,$ $A=\bmatrix{A_1\\ A_2}\in \R^{m\times n},$ a SPD matrix $\mathcal{S}\in \R^{n\times n}$ and $\bm{r}=[r_1^T,r_2^T,r_3^T]^T.$
                \State \textbf{Output:} $\bm{w}=[w_1^T,w_2^T,w_3^T]^T.$
                \State \textbf{Steps:}
                \State  Solve $(S+A_1^TA_1)z_1=r_2$  for $z_1;$
                \State  Solve $(S+A_1^TA_1)z_2=A_1^Tr_1$ for $z_2$;
                \State  Compute  $w_3=r_3-A_2(z_1+ z_2)$:
                \State  Solve $(S+A_1^TA_1)w_2=r_2+ A_1^Tr_1$ for $w_2;$
                \State  Compute $w_1=r_1-A_1w_2.$ 
			
   \end{algorithmic}
		\end{algorithm}

The above factorization yields an efficient algorithm for solving the generalized residual equation (2.6), which we present as Algorithm \ref{alog2}.

On the other hand, we have the following decomposition:
        \begin{align}
\mathscr{P}_{\text{RBRS}}=\bmatrix{I & \bf 0& \bf 0\\ -A_1^T&I &\bf 0\\ \bf 0 &A_2 (S+A_1^TA_1)^{-1} & I} \bmatrix{I & A_1& \bf 0\\ \bf 0 &S+A_1^TA_1 &\0\\ \bf 0 & \bf 0 & I}.
\end{align}
Using the above decomposition, we obtain Algorithm \ref{alog3}, which must be performed at each step of the preconditioned GMRES method.

\begin{remark}
   In Algorithm~\ref{alog2}, three linear subsystems are solved with the same coefficient matrix
$X=S+\frac{1}{\alpha +1}A_1^TA_1$. Since $X$ is SPD, these systems can be efficiently solved using either the Cholesky factorization or a preconditioned conjugate gradient (PCG) method. It is worth noting that the Cholesky factorization of $X$ needs to be performed only once, thereby reducing the overall computational cost. On the other hand, Algorithm~\ref{alog3} involves solving three linear subsystems with the coefficient matrix $S+A_1^TA_1$ instead of $ X$. Since $S+A_1^TA_1$ is also SPD, we can follow a similar method for the efficient implementation of Algorithm~\ref{alog3}.
\end{remark}
\subsection{Spectral Analysis of the BRS Preconditioned Matrix $\mathscr{P}_{\text{BRS}}^{-1}\mathcal{K}$} 
In this subsection, we discuss the spectral distribution of the both preconditioned matrices.

\begin{theorem}\label{th1:Eigdist}
    Let $A_1 \in \mathbb{R}^{n \times n}$ be a full column rank matrix and $S$ be a SPD matrix, and $\alpha>0.$
Suppose $\mu$ is an eigenvalue of the preconditioned matrix $\mathscr{P}^{-1}_{\text{BRS}} \mathcal{K}$. 
Then, 
\[
|\mu - 1| < 1,
\]
i.e., the spectrum of $\mathscr{P}^{-1}_{\text{BRS}} \mathcal{K}$ is entirely contained within a circle centered at $(1,0)$ 
with radius strictly less than $1$.
\end{theorem}

To further illustrate the effectiveness of our preconditioner \(\mathscr{P}\), we establish explicit and tight spectral bounds for the preconditioned matrix. 

Let $\lambda$ be an eigenvalue of the preconditioned matrix $\mathscr{P}_{\text{BRS}}^{-1}\mathcal{K}.$ Then, we have $\mathcal{K}\bm{z}=\lambda \mathscr{P}_{\text{BRS}}\bm{z},$ where $\bm{z}$ is the corresponding eigenvector. Then,  we obtain the following linear system of equations:
\begin{align}\label{eq:evd1}
       &  u+ A_1v=(\alpha +1)\lambda u+ \lambda A_1v,\\ \label{eq:evd2}
         &- A_1^Tu+A_2^T w=-\lambda A_1^Tu+\lambda Sv,\\ \label{eq:evd3}
         &A_2v+ w=\lambda A_2v+\lambda (\alpha+1)w.
     \end{align}

     \begin{theorem}
        Let $\lambda$ be an eigenvalue of the preconditioned matrix $\mathscr{P}_{\text{BRS}}^{-1}\mathcal{K}.$ Then $\lambda \neq 1$ and $\lambda =\frac{1}{\alpha +1}$ is an eigenvalues of $\mathscr{P}_{\text{BRS}}^{-1}\mathcal{K}$ with corresponding eigenvector $(\0,\0,w^T)^T,$ where $w\in \mathrm{Null}(A_2^T).$ The of the eigenvalues satisfies the quadratic equation: \[\bigl((\alpha+1)s + p\bigr)\lambda^2
+ \bigl(q - 2p - s\bigr)\lambda
+ (p - q) = 0,\]
where $s= \frac{v^TSv}{v^Tv}, $ $p= \frac{v^TA_1 v}{v^Tv}$ and $q= \frac{v^TA_2 v}{v^Tv}$.
     \end{theorem}
     \begin{proof}
         Let $\lambda=1.$ Then, from \eqref{eq:evd1} $\alpha u=\0,$ which gives $u=\0.$ From \eqref{eq:evd2} and \eqref{eq:evd3}, we obtain $v=0$ and $w=0,$ therefore $\bm{z}=\0,$ which is a contradiction to the fact that $\bm{z}$ is an eigenvector. Hence, $\lambda\neq 1.$

\noindent Next, consider $v=\0.$ Then, from \eqref{eq:evd1} and \eqref{eq:evd3}, we have $u=(\alpha+1) \lambda u$ and $w=(\alpha+1) \lambda w,$ respectively. If $u=\0,$ \eqref{eq:evd2} gives $A_2^Tw=\0,$ which implies $w\in \text{Null}(A_2^T).$ Hence, we obtain $\lambda=\frac{1}{\alpha+1}$ is an eigenvalue of the preconditioned matrix with corresponding eigenvector $(\0,\0,w^T)^T,$ where $w\in \text{Null}(A_2^T).$

   Consider $\v\neq \0$ and $\lambda\neq \frac{1}{\alpha+1}.$      From \eqref{eq:evd1} and \eqref{eq:evd3}, we have
\[
\bigl(1-(\alpha+1)\lambda\bigr)u + (1-\lambda)A_1 v = 0 
\]
and \[\bigl(1-(\alpha+1)\lambda\bigr)w + (1-\lambda)A_2 v = 0.\]
Since $1-(\alpha+1)\lambda \neq 0$ (i.e., $\lambda \neq \tfrac{1}{\alpha+1}$), it follows that
\[
u = \frac{\lambda-1}{\,1-(\alpha+1)\lambda\,}\, A_1 v~~\text{and}~~w = \frac{\lambda-1}{\,1-(\alpha+1)\lambda\,}\, A_2 v.
\]
Substituting these values in \eqref{eq:evd2}, we obtain:
\begin{align}
    \lambda Sv=\frac{(\lambda-1)^2}{\,1-(\alpha+1)\lambda\,}\, A_1 v+\frac{\lambda-1}{\,1-(\alpha+1)\lambda\,}\, A_2 v.
\end{align}
The above can be rewritten as:
\begin{align}
    \lambda s=\frac{(\lambda-1)^2}{\,1-(\alpha+1)\lambda\,}\, p+\frac{\lambda-1}{\,1-(\alpha+1)\lambda\,}\, q.
\end{align}
By simplifying, we obtain:
\begin{equation}
\bigl((\alpha+1)s + p\bigr)\lambda^2
+ \bigl(q - 2p - s\bigr)\lambda
+ (p - q) = 0.
\end{equation}
This completes the proof.
     \end{proof}

The preconditioner $\mathscr{P}_{\text{RBRS}}$ can be decompose as follows:
 \begin{align}
\mathscr{P}_{\text{RBRS}}=\bmatrix{I & \bf 0& \bf 0\\ -A_1^T&I &\bf 0\\ \bf 0 &A_2 (S+A_1^TA_1)^{-1} & I} \bmatrix{I & A_1& \bf 0\\ \bf 0 &S+A_1^TA_1 &\0\\ \bf 0 & \bf 0 & I}.
\end{align}

To further illustrate the effectiveness of our preconditioner \(\mathscr{P}_{\text{RBRS}}\), we investigate the eigenvalue properties of the RBRS  preconditioned matrix. 
\begin{theorem}
    Let $A_1 \in \mathbb{R}^{n \times n}$ be a full column rank matrix, $S$ be a SPD matrix, and $\alpha>0.$
Then, all the eigenvalues of the preconditioned matrix $\mathscr{P}_{\text{RBRS}}^{-1} \mathcal{K}$ are real and belong to the interval $(0,1]$. More precisely, it has $1$ as its eigenvalues with multiplicity at most $p,$ and the remaining eigenvalues are the same as those of the symmetric matrix  
\begin{align*}
\widehat{H}=
\bmatrix{
I-S^{1/2}Y^{-1}S^{1/2} & S^{1/2}Y^{-1}A_2^T \\
A_2Y^{-1}S^{1/2} & I-A_2Y^{-1}A_2^T
}
\end{align*}
that satisfy
\[
0< \lambda_i(\widehat H)\le 1.
\]
Hence,
\[
0<\lambda_i\!\left(\mathscr{P}_{\mathrm{RBRS}}^{-1}\mathcal K\right)\le 1.
\]
%
\end{theorem}

\begin{proof}
By direct computation, the preconditioned matrix $\mathscr{P}_{\text{RBRS}}^{-1}\mathcal{K}$ has the following form:
\begin{align}
\mathscr{P}_{\text{RBRS}}^{-1}\mathcal{K}
&=
\bmatrix{
I-A_1Y^{-1}A_1^T & -A_1Y^{-1} & \0\\
Y^{-1}A_1^T & Y^{-1} & \0\\
-A_2Y^{-1}A_1^T & -A_2Y^{-1} & I
}
\bmatrix{
I & A_1 &\0\\
-A_1^T & \0 &A_2^T \\
\0 & A_2 & I
}\\
&=
\bmatrix{
I &(I-A_1Y^{-1}A_1^T)A_1 & -A_1Y^{-1}A_2^T\\
\0 & Y^{-1}A_1^TA_1 & Y^{-1}A_2^T \\
\0 & A_2(I-Y^{-1}A_1^TA_1)& I-A_2Y^{-1}A_2^T
},
\end{align}
where $Y=S+A_1^TA_1.$

Therefore, the preconditioned matrix $\mathscr{P}_{\text{RBRS}}^{-1}\mathcal{K}$ has $1$ as its eigenvalue with multiplicity at least $p$, and the rest of the eigenvalues are the eigenvalues of the following matrix:
\begin{align*}
H=
\bmatrix{
Y^{-1}A_1^TA_1 & Y^{-1}A_2^T \\
A_2(I-Y^{-1}A_1^TA_1)& I-A_2Y^{-1}A_2^T
}.
\end{align*}
The above matrix can further be expressed as:
\begin{align*}
H=
\bmatrix{
I-Y^{-1}S & Y^{-1}A_2^T \\
A_2Y^{-1}S& I-A_2Y^{-1}A_2^T
},
\end{align*}
as $Y^{-1}S+Y^{-1}A_1^TA_1=I.$

By taking the similarity transformation, we obtain:
\begin{align*}
\widehat{H}
&=
\bmatrix{
S^{1/2} & \0\\
\0 & I
}
\bmatrix{
I-Y^{-1}S & Y^{-1}A_2^T \\
A_2Y^{-1}S& I-A_2Y^{-1}A_2^T
}
\bmatrix{
S^{-1/2} & \0\\
\0 & I
}\\
&=
\bmatrix{
I-S^{1/2}Y^{-1}S^{1/2} & S^{1/2}Y^{-1}A_2^T \\
A_2Y^{-1}S^{1/2}& I-A_2Y^{-1}A_2^T
}.
\end{align*}
Observe that the matrix $\widehat{H}$ is symmetric, and thus has all eigenvalues real. Since $H$ is similar to $\widehat{H}$, all eigenvalues of $H$ are also real.

\noindent Next, define
\[
B=
\bmatrix{
S^{1/2}Y^{-1/2}\\
A_2Y^{-1/2}
}.
\]
Then,
\[
\widehat H = I-BB^T,
\]
since
\[
BB^T=
\bmatrix{
S^{1/2}Y^{-1}S^{1/2} & S^{1/2}Y^{-1}A_2^T\\
A_2Y^{-1}S^{1/2} & A_2Y^{-1}A_2^T
}.
\]
Hence,
\[
\lambda_i(\widehat H)=1-\lambda_i(BB^T).
\]
Moreover, the nonzero eigenvalues of $BB^T$ coincide with those of
\[
B^TB
=
Y^{-1/2}(S+A_2^TA_2)Y^{-1/2}.
\]
Therefore,
\[
1-\lambda_{\max}\!\left(
Y^{-1/2}(S+A_2^TA_2)Y^{-1/2}
\right)
\le
\lambda_i(\widehat H)
\le 1.
\]
Since $A^TJA$ is SPD, we have
\[
A_1^TA_1-A_2^TA_2 \succ 0,
\]
which implies
\[
S+A_2^TA_2 \preceq S+A_1^TA_1=Y.
\]
Therefore,
\[
Y^{-1/2}(S+A_2^TA_2)Y^{-1/2}\prec I,
\]
and hence
\[
0\le
\lambda_i\!\left(
Y^{-1/2}(S+A_2^TA_2)Y^{-1/2}
\right)
< 1.
\]
Consequently,
\(
0< \lambda_i(\widehat H)\le 1,
\)
since
\[
\lambda_i(\widehat H)
=
1-
\lambda_i\!\left(
Y^{-1/2}(S+A_2^TA_2)Y^{-1/2}
\right).
\]
This completes the proof.
\end{proof}
\section{Numerical Experiments}\label{SEC5}

To evaluate the effectiveness and robustness of the proposed BRS and RBRS preconditioners for solving the ILS problem, we conduct several numerical experiments within the Krylov subspace iterative framework. Specifically, we compare the proposed BRS and RBRS preconditioned GMRES methods with the standard GMRES method \cite{saad1986gmres} without preconditioning and several preconditioned GMRES methods with
two block splitting (BS) preconditioners \cite{BSP2025OILS} given by:
\begin{align*}
    BS_1= \bmatrix{ I & \bf 0 &\bf 0\\ \bf 0& P &\bf 0\\ \bf 0 &\bf 0 &I} ~\text{and}~~ BS_2= \bmatrix{ I & \bf 0 &\bf 0\\ \bf 0& P &A_2^T\\ \bf 0 &\bf 0 &I},
\end{align*}
where $P=A_1^TA_1.$

For all iterative methods, the initial approximation is chosen as
\(
\bm{u}^{(0)} =
\bmatrix{
\bm{v}^{T},\,
\bm{x}_0^{T} ,\,
\bm{p}^{T}
},
\)
where \(\bm{x}_0=\bf{0}\) and the initial residual components are defined by
\[
\bmatrix{
\bm{v}^{T} \\
\bm{p}^{T}
}
= \bm{b}-A\bm{x}_0 .
\]
%
The iteration process is terminated when 
\begin{equation*}
    {\text{RES}} := 
    \frac{\| \mathcal{K}\bm{u}^{(k+1)} - \bm{b} \|_2}{\| \bm{b} \|_2} < 10^{-6},
\end{equation*}
or when the maximum number of iterations exceeds $5000$.  We compute the relative error in the approximate solution using the following:
\[ \text{ERR}=\frac{\|\bm{x}_{\mathrm{appx}}- \bm{x}_{\mathrm{exact}}\|_2}{\|\bm{x}_{\mathrm{exact}}\|_2}.\]
The component $\bm{x}_{\mathrm{appx}}$ extracted from $\bm{u}$ represents the current approximate solution to the ILS problem \eqref{ILS1}. 



All numerical experiments are performed in \textsc{MATLAB} R2024a 
on a \textsc{Windows 11} operating system, 
using an \textsc{Intel(R) Core(TM) i7-9700T CPU @ 2.00 GHz, 1992 Mhz, 8 Core(s)} 
with \textsc{16 GB RAM}.

\begin{example}\label{exam1}
Let $\epsilon = 10^{-4}$, and define
\[
\widetilde{X} = Y 
\bmatrix{
D \\[4pt]
\bf 0
}
Z^{T} \in \mathbb{R}^{p \times n},
\]
where $Y, Z$ are given orthogonal matrices and 
\[
D = \mathrm{diag}(1, 1/2, \ldots, 1/n).
\]
Further, let
\[
X = \widetilde{X} + \epsilon E\in \R^{p\times n}, 
\qquad 
d = \widetilde{X}\mathbf{1}_n + \epsilon f,
\]
where $\mathbf{1}_n$ denotes an $n \times 1$ column vector of all ones, and $E$, $f$ are given error matrix and vector generated using the MATLAB command ``{\tt rand}($\cdot$)". 
If the matrix $X^{T}X - \sigma_{n+1}^{2}I_n$ is positive definite, the solution of the total least-squares problem associated with $X, d$ is given by
\[
\bm{x}_{\mathrm{TLS}} = (X^{T}X - \sigma_{n+1}^{2}I_n)^{-1} X^{T} d,
\]
where $\sigma_{n+1}$ is the smallest singular value of $\bmatrix{X & d}$. 

\begin{table}[htbp]
\centering
\caption{Numerical results for different preconditioned GMRES methods for Example \ref{exam1}}
\begin{tabular}{ccccccc}
\toprule
$(p, n)$ & &GMRES & $\mathrm{BS}_1$ & $\mathrm{BS}_2$ & BRS & RBRS\\
\midrule

\multirow{4}{*}{$(64, 32)$}
 & IT  & 48 & 2 & 2 & 2 & 2\\
 & CPU & 0.0086 & 0.0071 & 0.0044 & 0.0014 & 0.0147\\
 & RES & 2.8595E-07 & 4.8300E-12 & 2.5003E-10 & 1.2642E-10 & 1.0732E-10\\
 & ERR & 1.5048E-07 & 6.9112E-10 & 1.4183E-14 & 6.4811E-11 & 2.7610E-12\\

\midrule

\multirow{4}{*}{$(128, 64)$}
 & IT  & 64 & 2 & 2 & 2 & 2\\
 & CPU & 0.0058 & 0.0095 & 0.0098 & 0.0042 & 0.0066\\
 & RES & 9.3649E-07 & 1.4520E-10 & 4.7491E-09 & 3.2035E-09 & 5.4468E-10\\
 & ERR & 2.1166E-06 & 1.2303E-08 & 2.0349E-13 & 4.9118E-10 & 2.5359E-11\\

\midrule

\multirow{4}{*}{$(256, 128)$}
 & IT  & 88 & 2 & 2 & 2 & 2\\
 & CPU & 0.0288 & 0.0314 & 0.0376 & 0.0042 & 0.0036\\
 & RES & 7.5591E-07 & 2.4410E-09 & 3.5777E-08 & 3.2035E-09 & 3.1789E-09\\
 & ERR & 7.0259E-06 & 1.7371E-07 & 6.4496E-12 & 4.9118E-10 & 2.9646E-10\\

\midrule

\multirow{4}{*}{$(1024, 512)$}
 & IT  & 146 & 4 & 3 & 2 & 2\\
 & CPU & 0.3550 & 1.1359 & 0.9255 & 0.0388 & 0.0347\\
 & RES & 9.1209E-07 & 8.2644E-11 & 3.1185E-10 & 1.1445E-07 & 1.1414E-07\\
 & ERR & 1.3303E-04 & 2.2243E-09 & 5.2883E-12 & 3.9236E-08 & 3.8228E-08\\

\midrule

\multirow{4}{*}{$(2048, 1024)$}
 & IT  & 180 & 4 & 3 & 2 & 2\\
 & CPU & 2.6536 & 8.2573 & 6.8436 & 0.1591 & 0.1559\\
 & RES & 9.2476E-07 & 2.5273E-08 & 4.9380E-08 & 7.0466E-07 & 7.0390E-07\\
 & ERR & 5.4275E-04 & 6.3509E-07 & 2.0309E-09 & 3.7968E-07 & 3.7775E-07\\

\midrule

\multirow{4}{*}{$(4096, 2048)$}
 & IT  & 210 & 6 & 4 & 3 & 3\\
 & CPU & 11.7411 & 82.7155 & 52.0565 & 1.1370 & 1.2117\\
 & RES & 9.5833E-07 & 8.7196E-08 & 8.5439E-08 & 5.8313E-08 & 5.8284E-08\\
 & ERR & 2.1779E-03 & 1.7271E-06 & 8.5051E-08 & 4.4007E-08 & 4.3935E-08\\

\bottomrule
\end{tabular}
\label{tab1:example1}
\end{table}
It is equivalent to solving the ILS problem with $q=n$ and 
\[
A = 
\bmatrix{
X \\[4pt]
\sigma_{n+1} I_n
}\in \R^{(p+n)\times n},
\qquad
b = 
\bmatrix{
d \\[4pt]
\bf 0
},
\qquad
\mathcal{J} = \mathrm{diag}(I_p, -I_n).
\]

Table~\ref{tab1:example1} presents the numerical results for the proposed BRS and RBRS   preconditioned GMRES methods, together with the unpreconditioned GMRES method and the $\mathrm{BS}_1$ and $\mathrm{BS}_2$ preconditioned GMRES methods. For BRS and RBRS preconditioners, we choose $S = 0.1 I_n$ and $\alpha=1e-05.$ The performance of each method is evaluated in terms of the number of iterations (IT), CPU time (in seconds), the relative residual (RES), and the relative error (ERR).

The numerical results demonstrate the effectiveness of the proposed 
$\mathrm{BRS}$ and $\mathrm{RBRS}$ preconditioners. Notice that the proposed BRS and RBRS preconditioned methods outperform all other methods, in terms of both IT and CPU. Compared with the unpreconditioned GMRES method, the proposed preconditioners substantially 
reduce the number of iterations and, in most cases, the computational time. 
In particular, $\mathrm{BRS}$ requires only $2$--$3$ iterations for all 
tested problem sizes, whereas GMRES requires between $48$ and $210$ 
iterations. Moreover, $\mathrm{BRS}$ produces small residuals and errors 
while maintaining a relatively low computational cost. For the largest 
problem of size $6132\times2048$, for example, $\mathrm{BRS}$ reduces the 
iteration count of GMRES from $210$ to only $3$ and the CPU time from $11.7411$ to 
$1.1370$ seconds, while for RBRS it is $1.2117$ seconds.

\begin{figure}[h!]\label{eigen:ex1}
    \centering
    \begin{minipage}{0.4\textwidth}
        \centering
        \includegraphics[width=\linewidth]{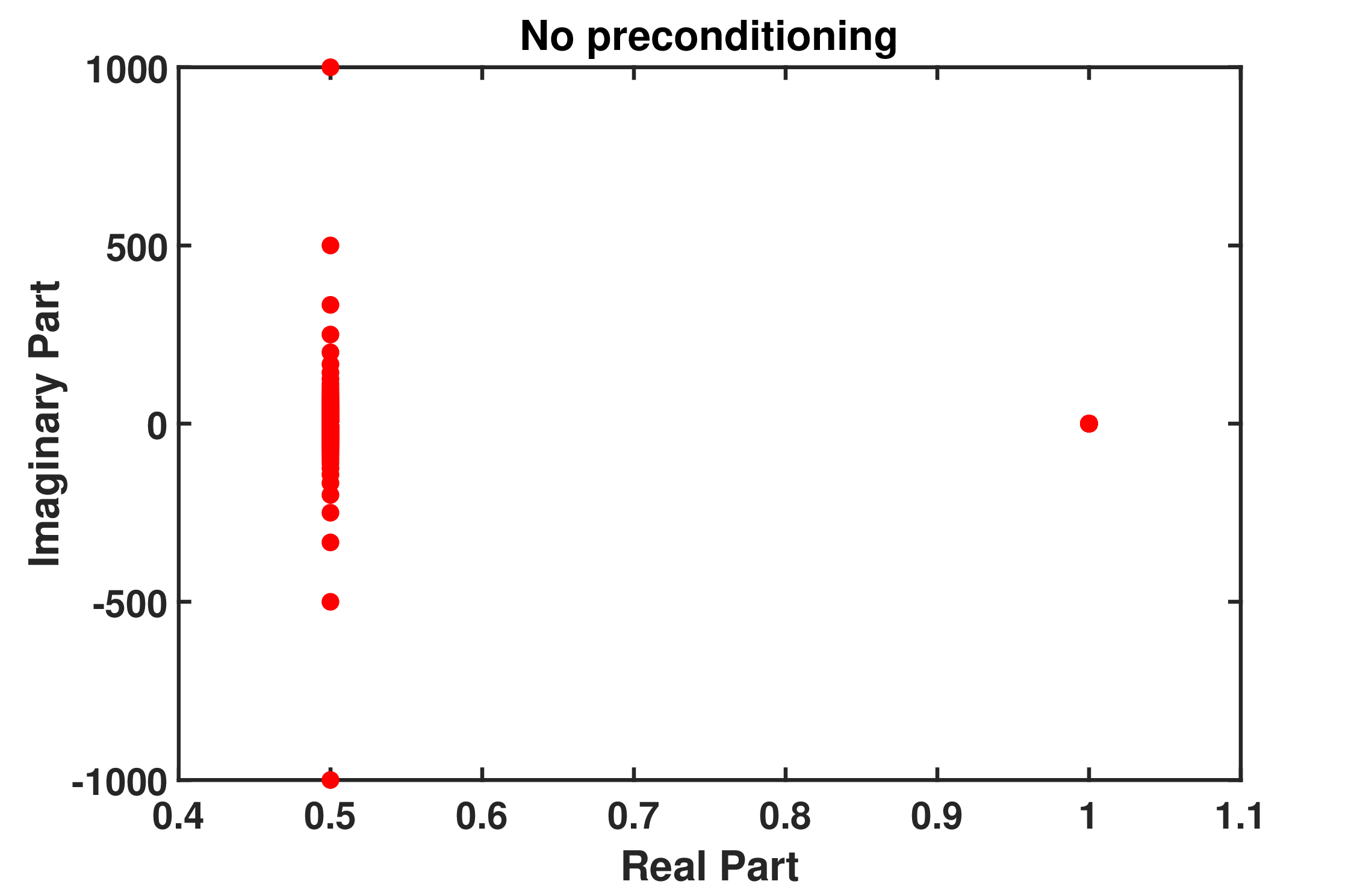}
    \end{minipage}
    \begin{minipage}{0.37\textwidth}
        \centering
        \includegraphics[width=\linewidth]{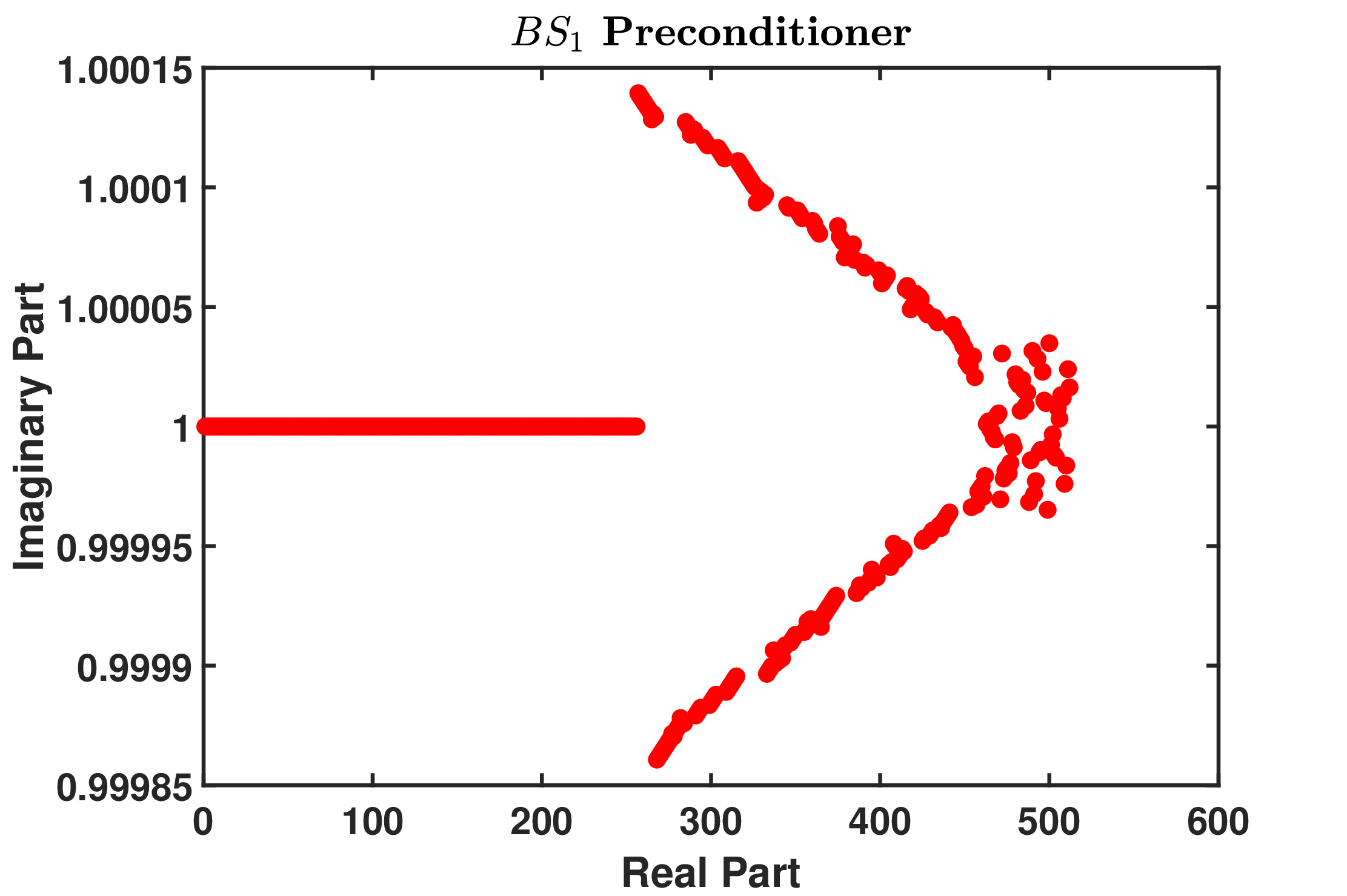}
    \end{minipage}
    \hfill
    \begin{minipage}{0.4\textwidth}
        \centering
        \includegraphics[width=\linewidth]{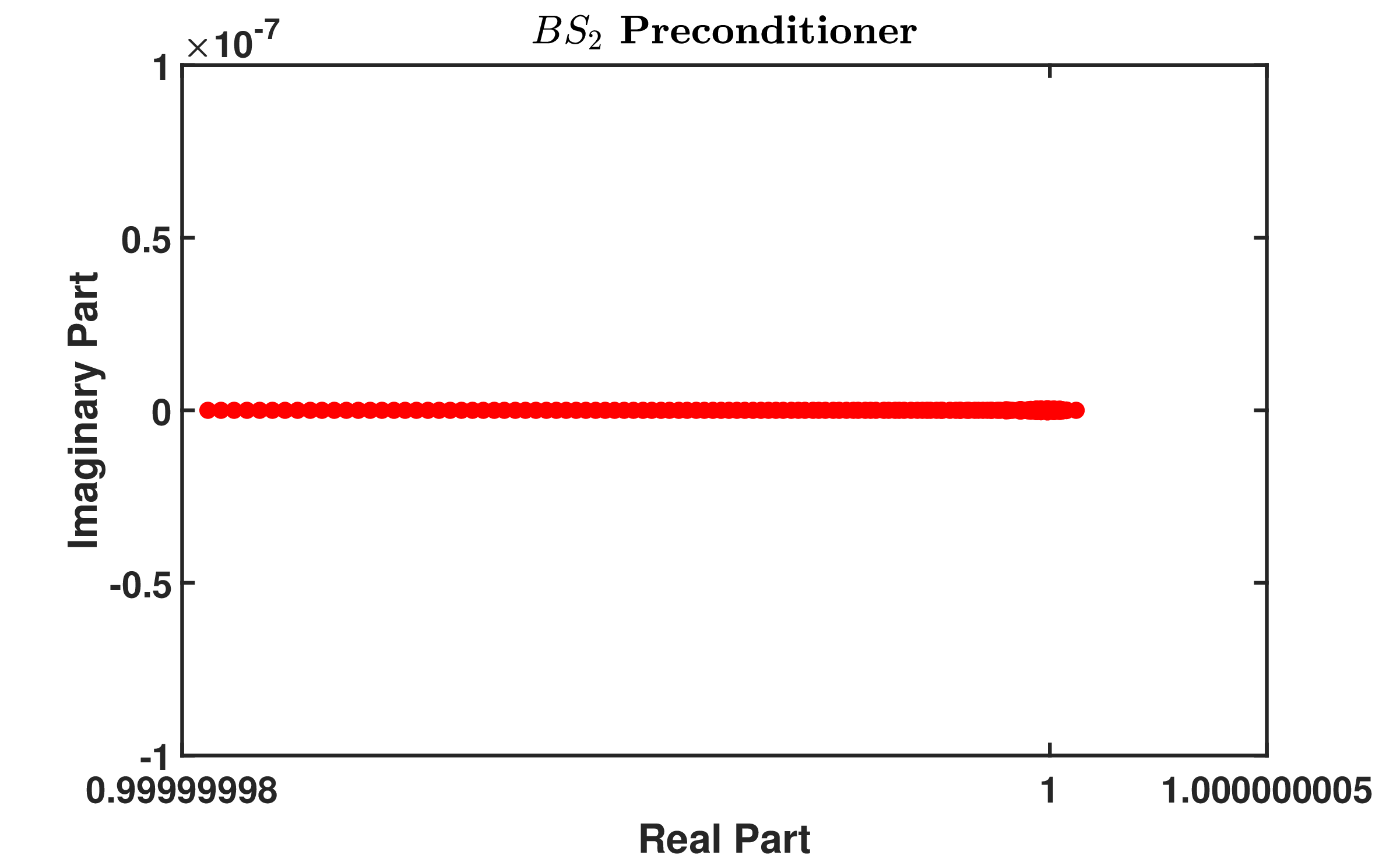}
    \end{minipage}
\begin{minipage}{0.39\textwidth}
        \centering
        \includegraphics[width=\linewidth]{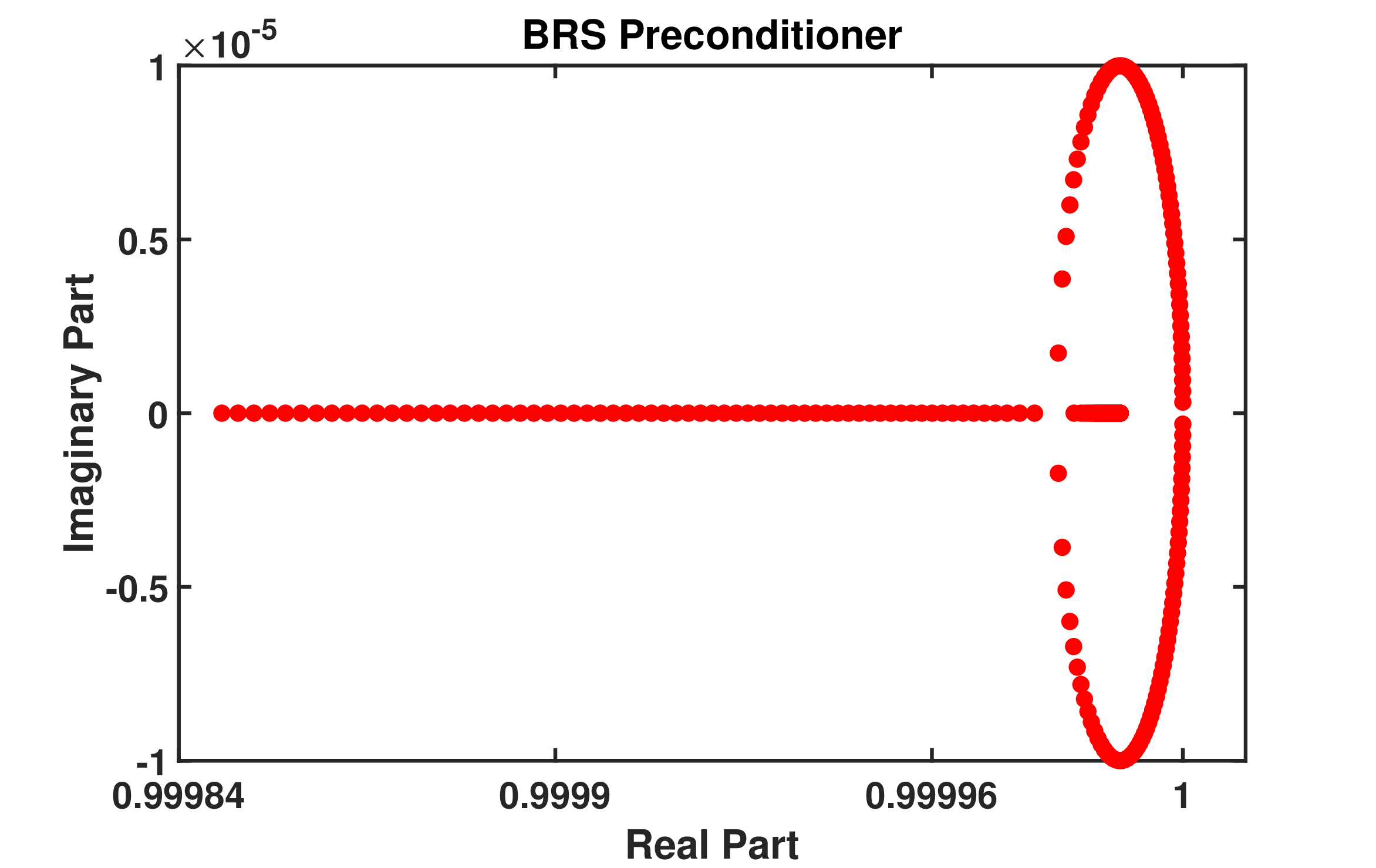}
    \end{minipage}
    \begin{minipage}{0.36\textwidth}
        \centering
        \includegraphics[width=\linewidth]{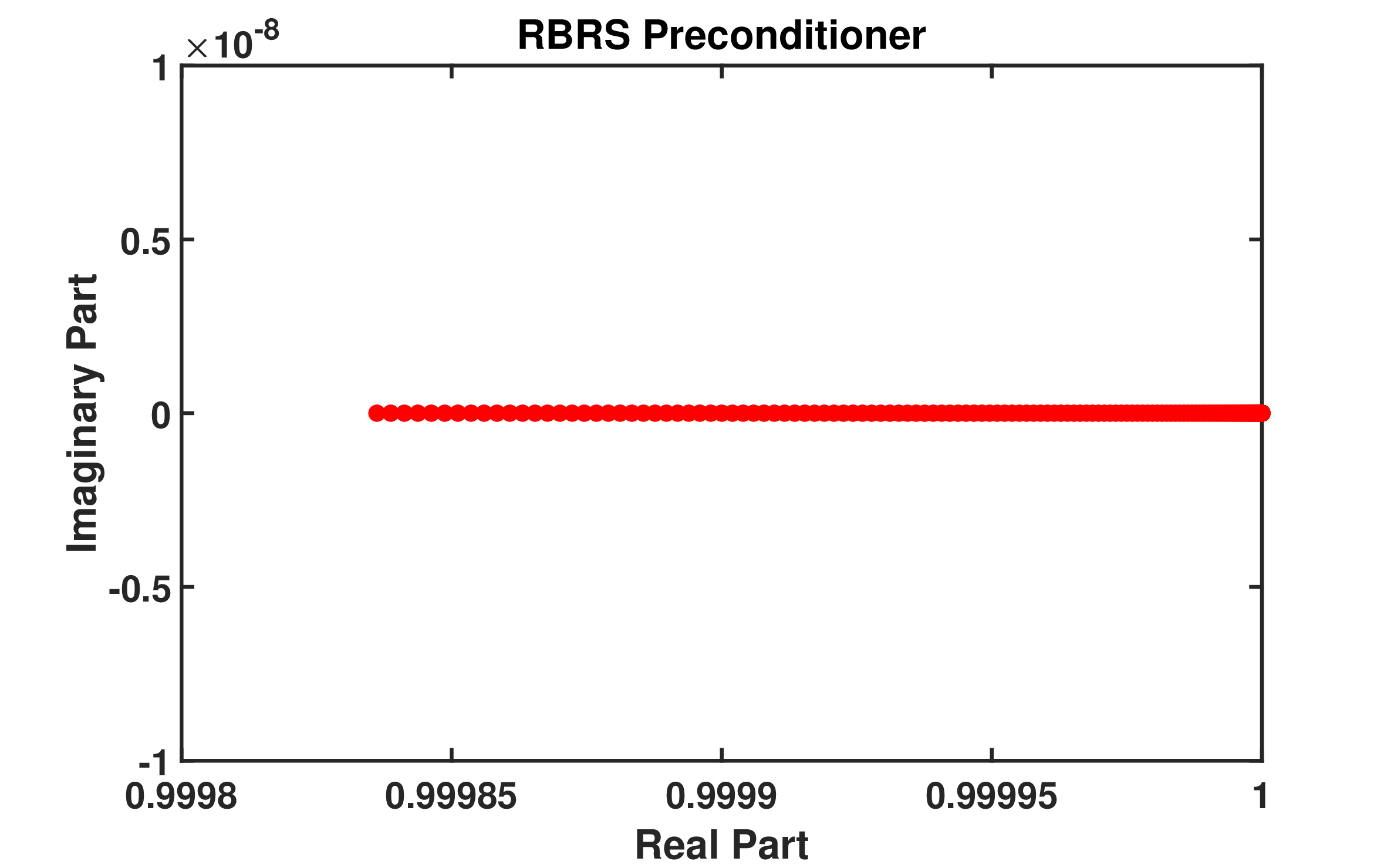}
    \end{minipage}
    \caption{Eigenvalue distributions of the original and preconditioned systems for Example \ref{exam1}}
\end{figure}

Figure~\ref{eigen:ex1} compares the eigenvalue distributions of the original and preconditioned systems. The eigenvalues of the unpreconditioned matrix are highly scattered, whereas the $\mathrm{BS}_1$ preconditioner only partially improves the spectral distribution. The proposed BRS and RBRS preconditioners significantly cluster the eigenvalues around the point $1$, with RBRS exhibiting the tightest clustering with all real eigenvalues equal to one or less than one. Such a favorable eigenvalue distribution is well known to accelerate the convergence of Krylov subspace methods, which is consistent with the iteration counts and CPU times reported in the numerical experiments.
\begin{table}[ht]
\centering
\caption{Estimated $2$-norm condition numbers of the original and preconditioned matrices for the generated matrix in Example \ref{exam1}.}
\label{tab1:cond_numbers}
\begin{tabular}{cccc}
\toprule
\textbf{ $(p,n)$} & $\kappa_2(\mathcal K)$ & $\kappa_2(\mathscr{P}_{\text{BRS}}^{-1}\mathcal K)$ & $\kappa_2(\mathscr{P}_{\text{RBRS}}^{-1}\mathcal K)$ \\
\midrule
$(64,32)$     & $9.2290\times10^{3}$ & $1.027759$   & $1.027991$ \\
$(128,64)$    & $2.0314\times10^{4}$ & $1.080281$   & $1.080613$ \\
$(256,128)$   & $2.4150\times10^{4}$ & $1.232846$   & $1.233950$ \\
$(1024,512)$  & $6.3886\times10^{4}$ & $3.736119$   & $3.812940$ \\
$(2048,1024)$ & $1.1950\times10^{5}$ & $1.622671\times10^{1}$ & $1.590080\times10^{1}$ \\
$(4096,2048)$ & $4.9270\times10^{5}$ & $1.257223\times10^{2}$ & $1.265465\times10^{2}$ \\
\bottomrule
\end{tabular}
\end{table}

\begin{figure}[h!]
    \centering
    \begin{minipage}{0.4\textwidth}
        \centering
        \includegraphics[width=\linewidth]{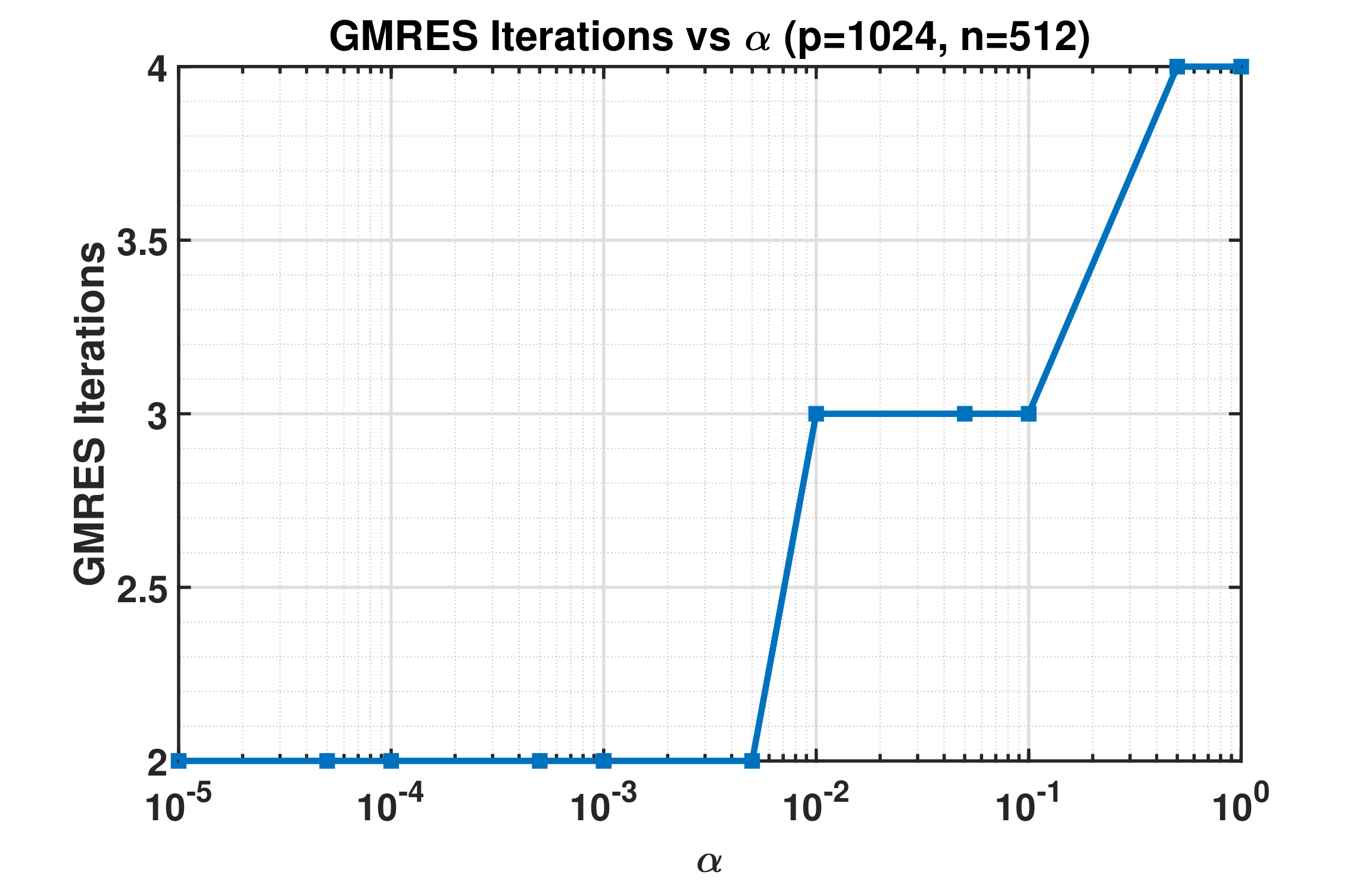}
    \end{minipage}
     \hfill
    \begin{minipage}{0.4\textwidth}
        \centering
        \includegraphics[width=\linewidth]{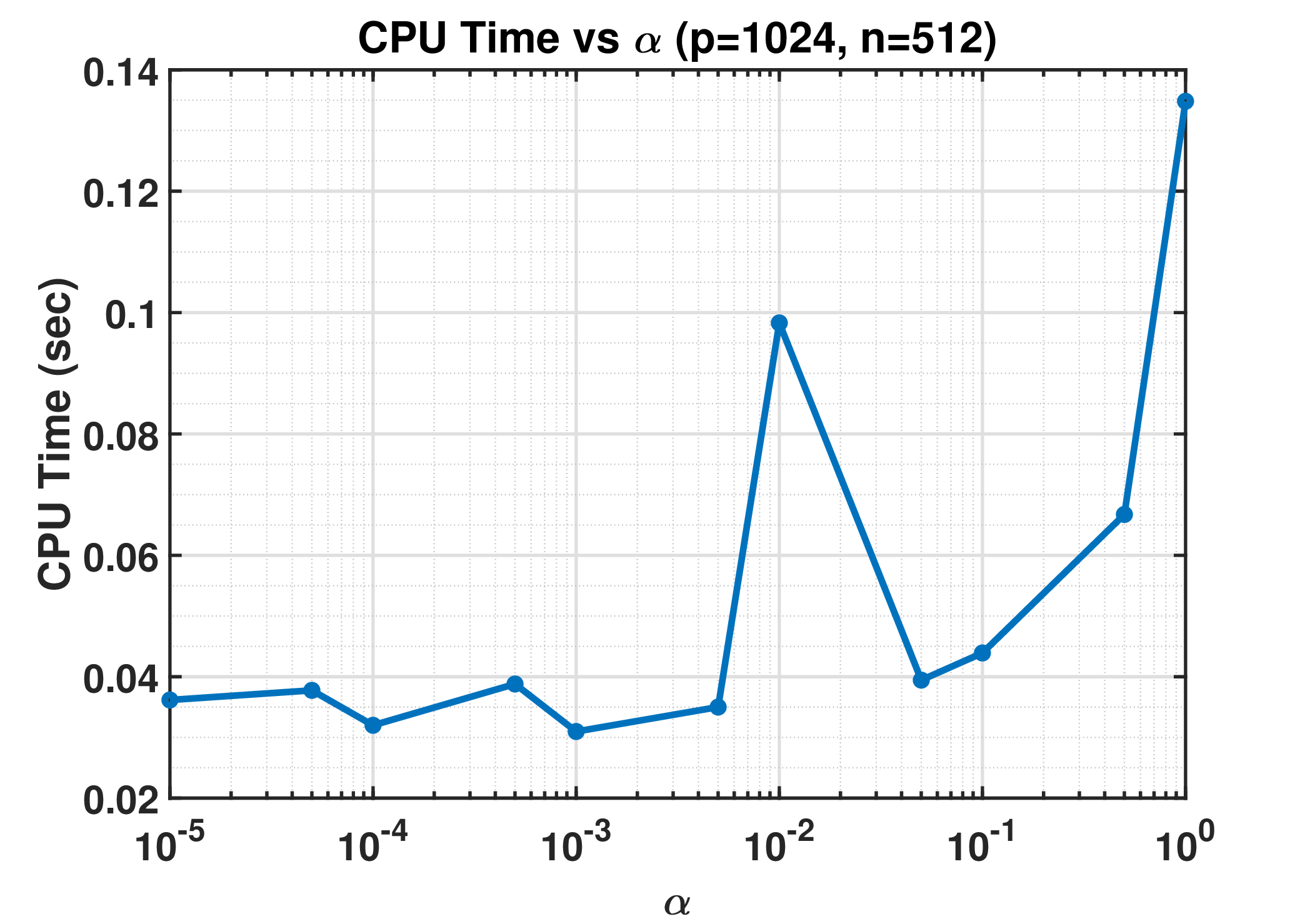}
    \end{minipage}
    \hfill
    \caption{Comparison of $\alpha$ versus IT and CPU time for Example \ref{exam1}}
    \label{alpha_cpu_it}
\end{figure}
The $\mathrm{RBRS}$ preconditioner also yields substantial reductions in 
the residual and error compared with GMRES. Overall, the results confirm that the proposed $\mathrm{BRS}$ and RBRS
preconditioners are particularly effective, providing rapid convergence and 
highly accurate solutions for all tested problem sizes.

Table~\ref{tab1:cond_numbers} reports the estimated $2$-norm condition numbers ($\kappa_2(A)=\|A\|_2 \|A^{-1}\|_2$) of the original coefficient matrix $\mathcal{K}$ and the corresponding preconditioned matrices $\mathscr{P}_{\text{BRS}}^{-1}\mathcal{K}$ and $\mathscr{P}_{\text{RBRS}}^{-1}\mathcal{K}$ for the test problems considered in Example~\ref{exam1}. These condition numbers are included to evaluate the effectiveness of the proposed preconditioners in improving the spectral properties of the linear system. As can be seen, the condition number of the original matrix increases rapidly with the problem size, indicating that the systems become increasingly ill-conditioned. In contrast, both the BRS and RBRS preconditioners reduce the condition numbers dramatically, bringing them close to one for the smaller problems and keeping them significantly lower than that of the original matrix even for the largest test cases. Moreover, the BRS and RBRS preconditioners exhibit nearly identical conditioning behavior, suggesting that the relaxed variant preserves the excellent conditioning properties of the original BRS preconditioner while maintaining robustness for large-scale problems.

Figure~\ref{alpha_cpu_it} illustrates the influence of the parameter $\alpha$ on the performance of the BRS preconditioner. It is observed that the GMRES iteration count remains nearly constant for smaller values of $\alpha$, indicating robust convergence. As $\alpha$ increases, both the iteration count and CPU time increase gradually, reflecting a slight deterioration in the efficiency of the preconditioner. Therefore, smaller values of $\alpha$ provide the best overall computational performance.
\end{example}

\begin{example}\label{example2}
We consider matrices $A_1$ derived from the Tolosa matrix, denoted as \texttt{TOLS340}, \texttt{TOLS1090}, \texttt{TOLS2000}, and \texttt{TOLS4000}. 
The Tolosa matrix arises from the stability analysis of an airplane in flight and has been studied at CERFACS in collaboration with the Aerospatiale Aircraft division. 
It is a sparse $5 \times 5$ block matrix of order $n = 90 + 5k$.

To ensure that the ILS problem (1.1) has a unique solution, we define
\[
A_2 = 0.3 \cdot I_{q \times n}, \quad 
b_1 = \text{rand}(p,1), \quad 
b_2 = \text{rand}(q,1),
\]
where $I_{q \times n}$ denotes the identity matrix of size $q \times n$.

In the implementation of the BRS and RBRS preconditioners, the positive parameter $\alpha$ should ideally be as small as possible. 
For BRS and RBRS preconditioners, we chose $S = 0.01 I_n$ and $\alpha=1e-05.$
\begin{table}[htbp]
\centering
\caption{Numerical results for different preconditioned GMRES methods for Example \ref{example2}}
\label{tab:exp2}
\begin{tabular}{ccccccccc}
\toprule
Size &  & GMRES & $\mathrm{BS}_1$ &$\mathrm{BS}_2$ &  BRS & RBRS\\
\midrule
\multirow{4}{*}{\texttt{TOLS340}}
 & IT  & 179 & 4 & 4 &   3 & 3\\
 & CPU & 0.9180 & 0.0203 & 0.0283 &   0.0345 & 0.0121\\
 & RES & 6.2694E-07 & 1.0404E-09 & 1.6282E-07 &  2.2871E-07 & 4.0133E-07\\
 & ERR & 1.0017E-06 & 1.1756E-09 & 2.1177E-06 &  4.3437E-07 & 2.5969E-07\\
\midrule
\multirow{4}{*}{\texttt{TOLS1090}}
 & IT  & 479 & 4 & 4 &   3 & 3\\
 & CPU & 9.6416 & 0.0225 & 0.0247 &   0.0200 & 0.0170\\
 & RES & 7.1667E-07 & 1.1453E-09 & 2.8069E-07 &  4.2175E-07 & 2.3733E-07\\
 & ERR & 5.6700E-07 & 7.0129E-10 & 2.3761E-06 &   8.0669E-07 & 1.3025E-06\\
\midrule
\multirow{4}{*}{\texttt{TOLS2000}}
 & IT  & 843 & 4 & 4 &   3 & 3\\
 & CPU & 41.0634 & 0.0308 & 0.0357 &   0.0163 & 0.0148\\
 & RES & 5.1050E-07 & 1.1838E-09 & 3.7580E-07 &   5.2954E-07 & 5.5227E-07\\
 & ERR & 2.1502E-07 & 4.2829E-10 & 2.4945E-06 &   4.8938E-07 & 7.2635E-07\\
\midrule
\multirow{4}{*}{\texttt{TOLS4000}}
 & IT  & 1643 & 4 & 4 &   3 & 3\\
 & CPU & 181.9980 & 0.0491 & 0.0568 &   0.0150 & 0.0512\\
 & RES & 2.1201E-07 & 7.4274E-09 & 4.9962E-07 &   7.5319E-07 & 7.4687E-07\\
 & ERR & 8.7837E-08 & 3.7201E-09 & 2.5667E-06 &   4.2183E-07 & 1.9738E-07\\
\bottomrule
\end{tabular}
\end{table}
\begin{figure}[h!]
    \centering
    \begin{minipage}{0.4\textwidth}
        \centering
        \includegraphics[width=\linewidth]{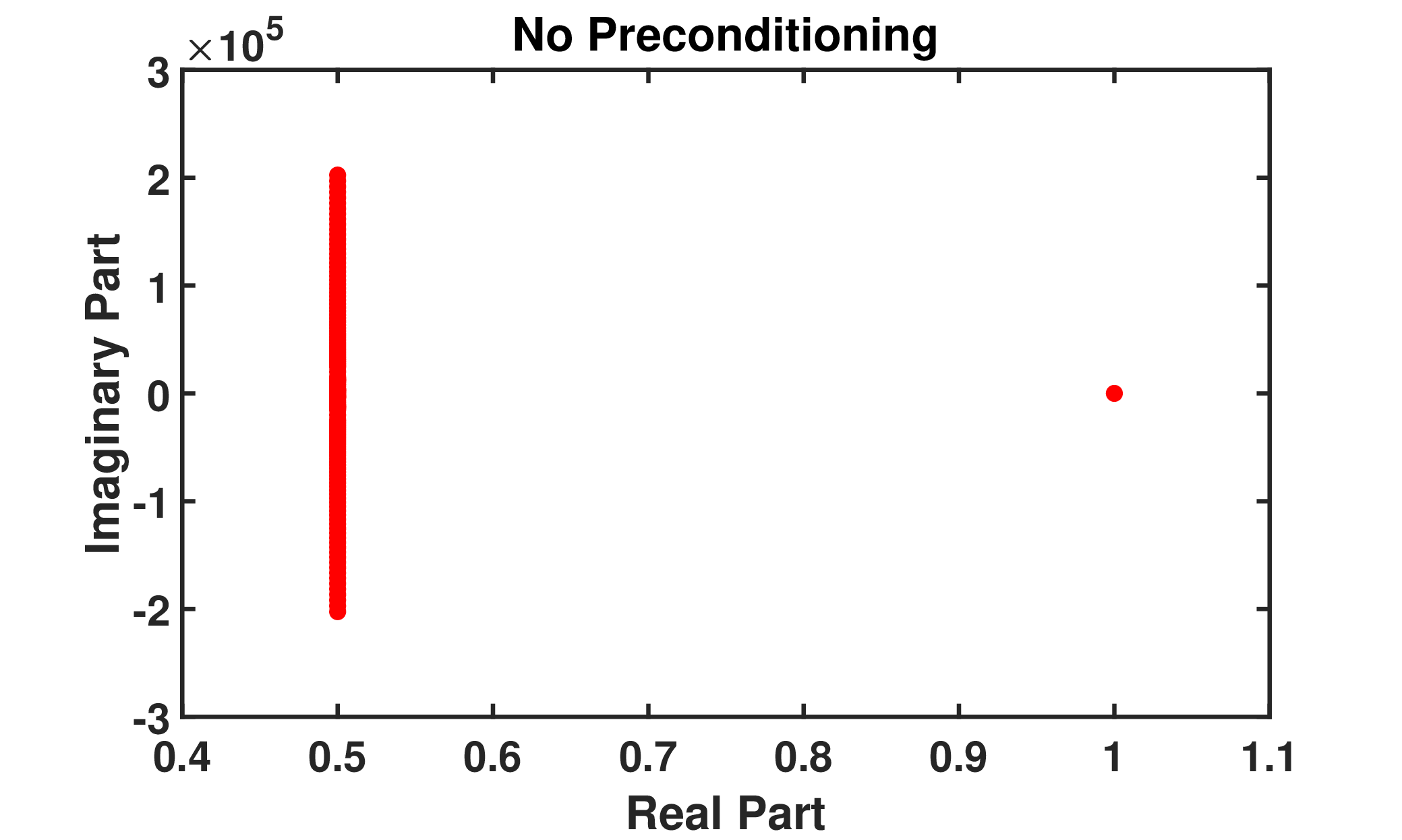}
    \end{minipage}
    \begin{minipage}{0.37\textwidth}
        \centering
        \includegraphics[width=\linewidth]{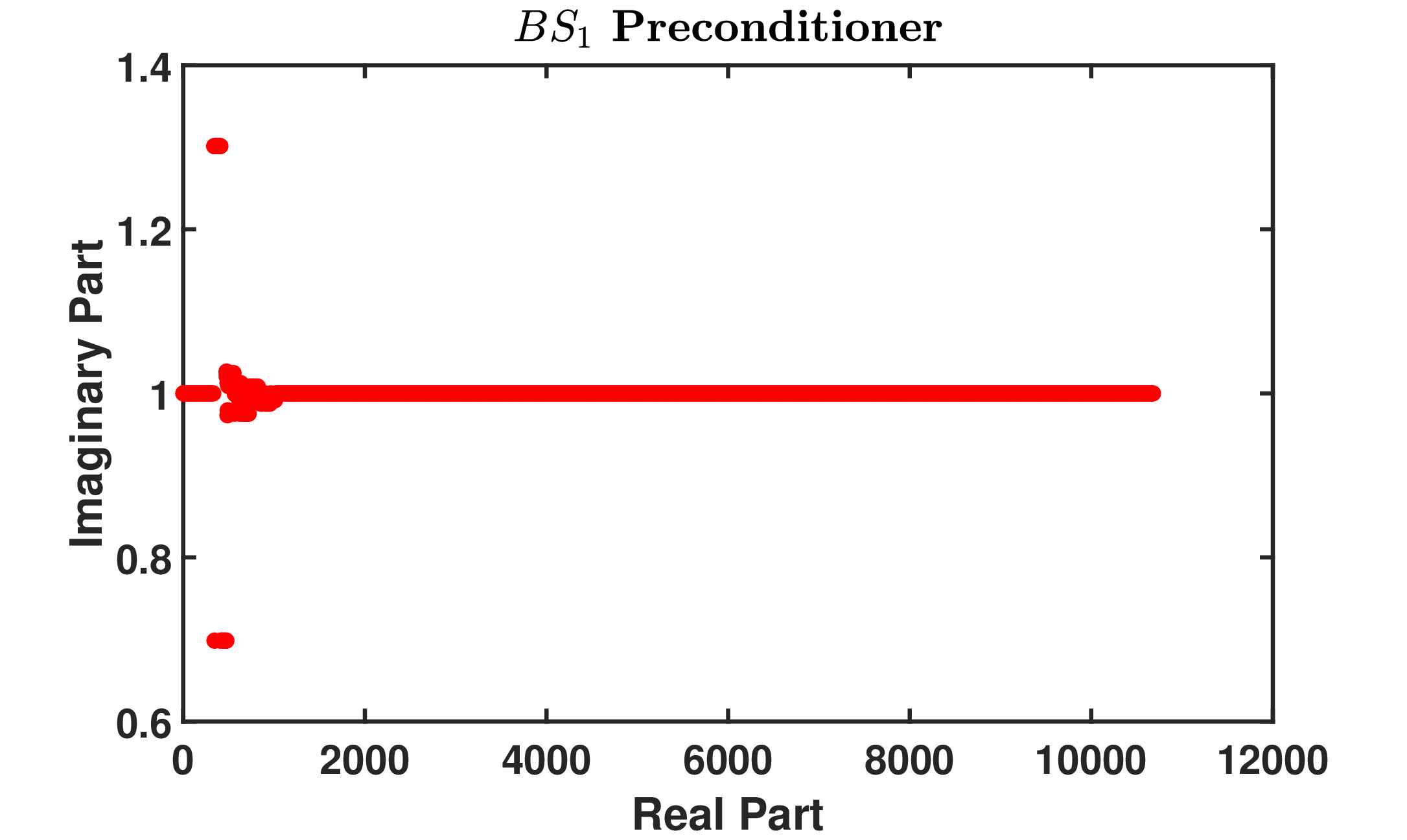}
    \end{minipage}
    \hfill
    \begin{minipage}{0.4\textwidth}
        \centering
        \includegraphics[width=\linewidth]{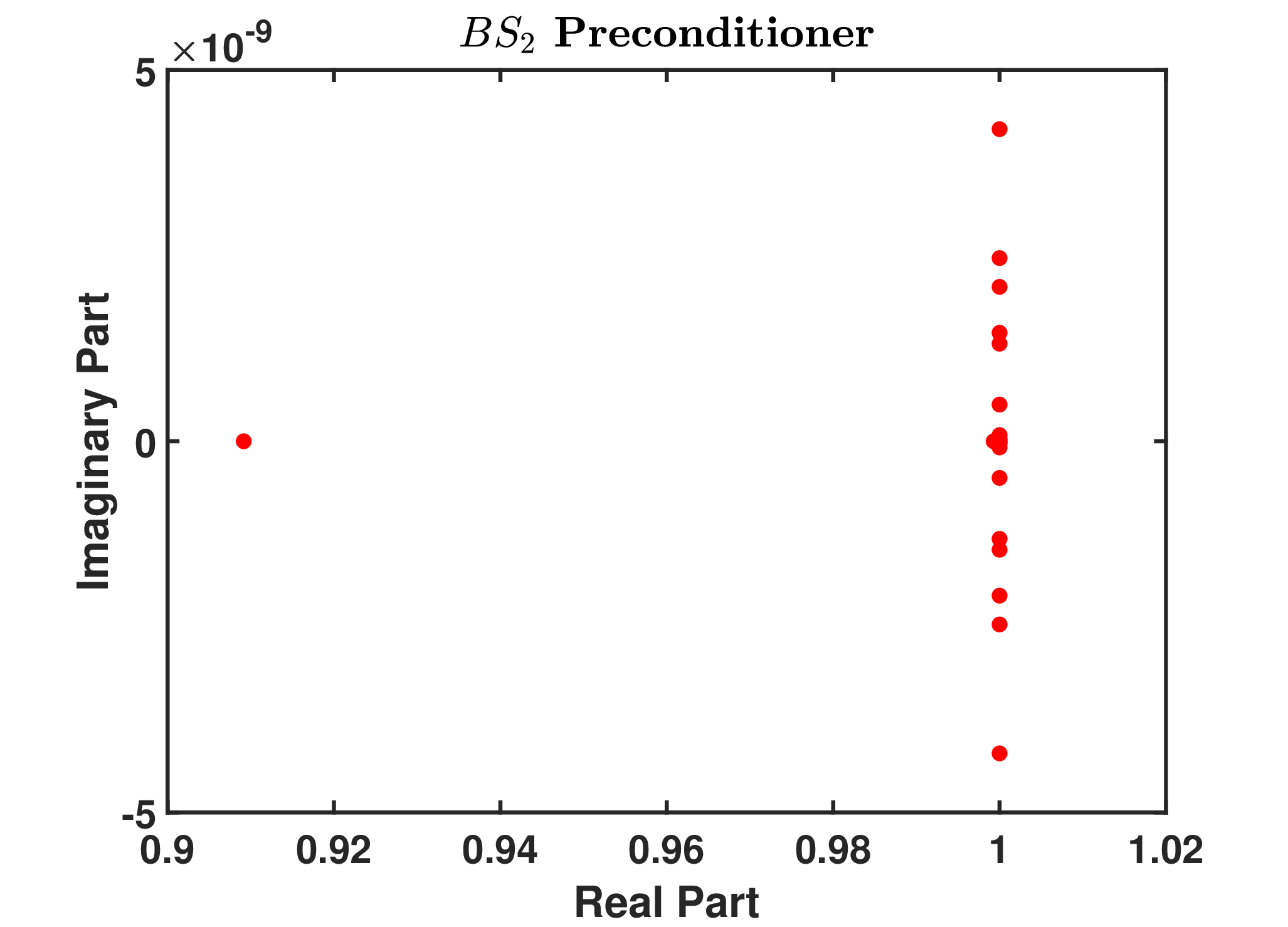}
    \end{minipage}
\begin{minipage}{0.39\textwidth}
        \centering
        \includegraphics[width=\linewidth]{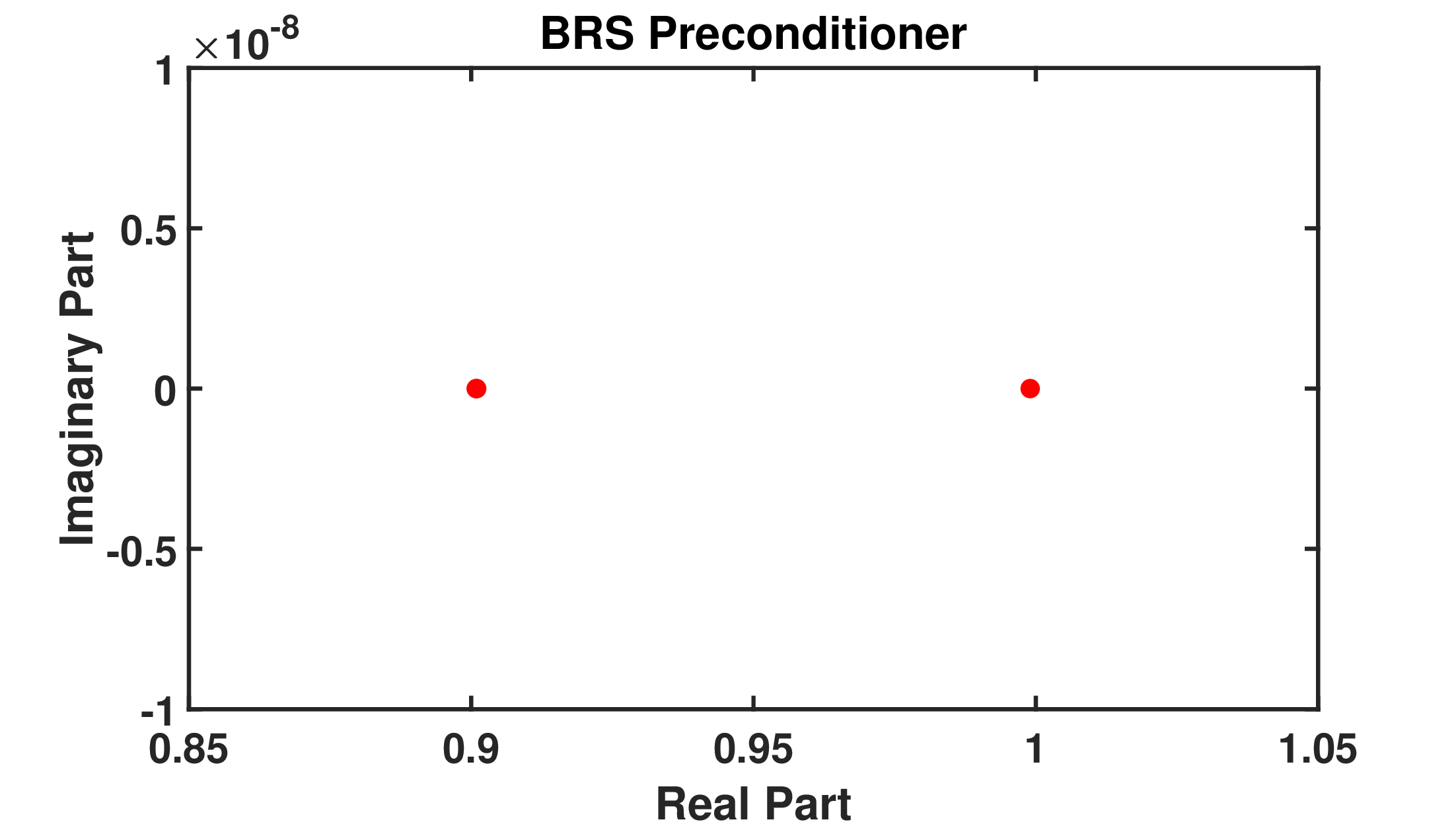}
    \end{minipage}
    \begin{minipage}{0.36\textwidth}
        \centering
        \includegraphics[width=\linewidth]{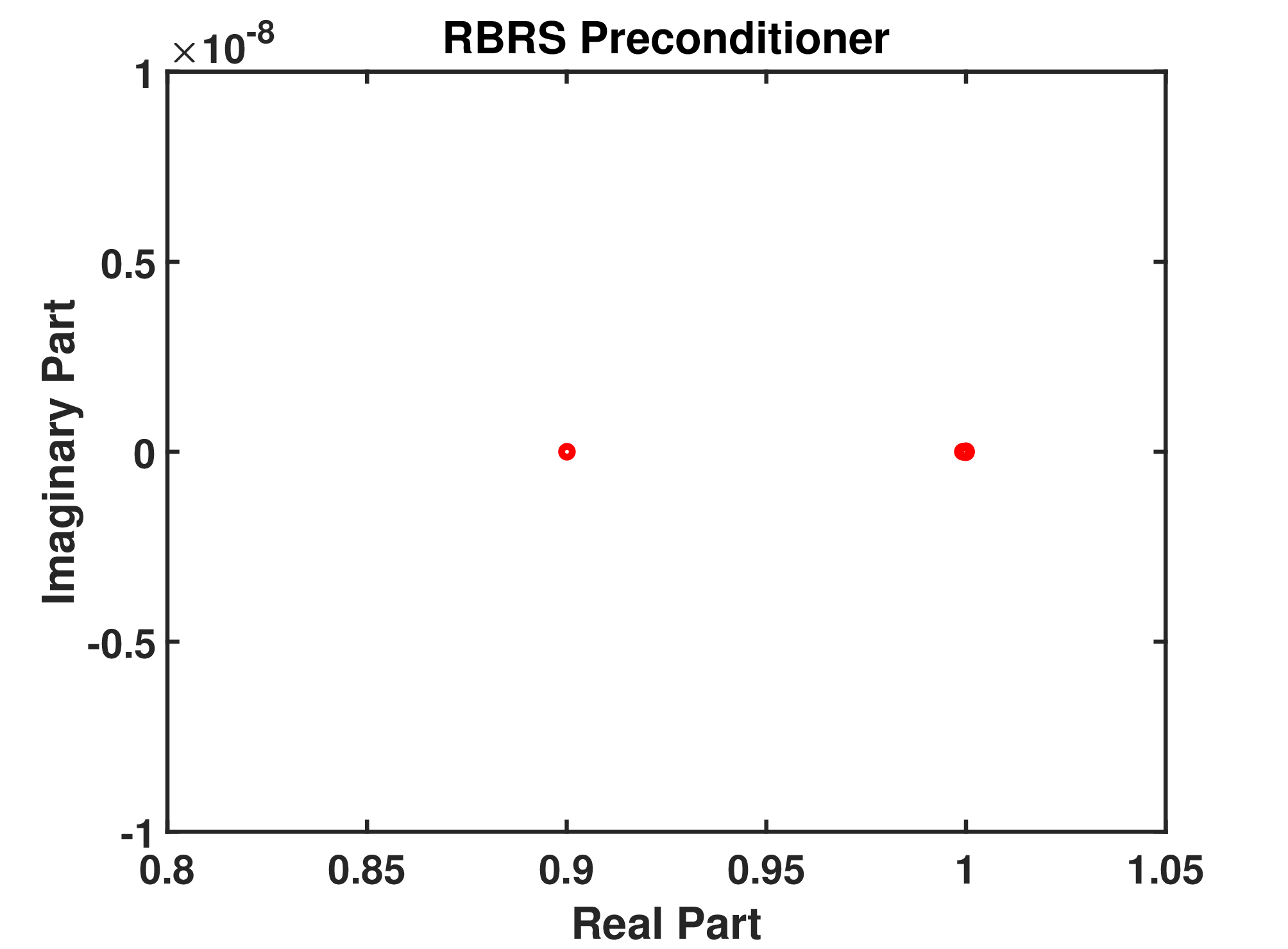}
    \end{minipage}
    \caption{Eigenvalue distributions of the original and preconditioned systems for Example \ref{example2}}
    \label{fig:eigenex2}
\end{figure}

Table~\ref{tab:exp2} compares the performance of the proposed BRS and RBRS preconditioners with GMRES and the recently proposed $\mathrm{BS}_1$ and $\mathrm{BS}_2$ preconditioners. It is observed that both proposed preconditioners outperform GMRES and the other two $\mathrm{BS}_1$ and $\mathrm{BS}_2$ preconditioned methods for all test problems. The RBRS preconditioner consistently requires lower CPU time than the other preconditioners in almost all cases. Although $\mathrm{BS}_1$ attains the smallest residuals and errors, the proposed BRS and RBRS achieve residuals and errors of the same order of magnitude while requiring only three GMRES iterations. In contrast, $\mathrm{BS}_2$ exhibits significantly larger errors than the proposed methods despite requiring four iterations. Overall, the proposed BRS and RBRS preconditioners provide a better balance between computational efficiency and numerical accuracy, with RBRS being the fastest preconditioner among all those considered.


\begin{table}[ht]
\centering
\caption{Estimated $2$-norm condition numbers of the original and preconditioned matrices for the Tolosa test problems in Example \ref{example2}.}
\label{tab:cond_numbers}
\begin{tabular}{lccc}
\toprule
\textbf{Matrix} & $\kappa_2(\mathcal K)$ & $\kappa_2(\mathscr{P}_{\text{BRS}}^{-1}\mathcal K)$ & $\kappa_2(\mathscr{P}_{\text{RBRS}}^{-1}\mathcal K)$ \\
\midrule
\tt TOLS340  & $5.5745\times10^{5}$ & $2.7763$ & $2.7773$ \\
\tt TOLS1090 & $5.0329\times10^{6}$ & $2.7757$ & $2.7756$ \\
\tt TOLS2000 & $1.6390\times10^{7}$ & $2.7756$ & $2.7755$ \\
\tt TOLS4000 & $6.4427\times10^{7}$ & $2.7755$ & $2.7754$ \\
\bottomrule
\end{tabular}
\end{table}
The eigenvalue distributions in Figure \ref{fig:eigenex2} clearly demonstrate the effectiveness of the proposed preconditioners. Without preconditioning, the eigenvalues are widely spread along a vertical line around 0.5, with a large imaginary component. In contrast, all the preconditioned systems (except $\mathrm{BS}_1$) exhibit strong eigenvalue clustering near 1 on the real axis, while their imaginary parts are reduced to nearly zero. In particular, the BRS and RBRS preconditioners produce eigenvalues that are essentially concentrated at 1, indicating a highly favorable spectral distribution. The $\mathrm{BS}_1$ and $\mathrm{BS}_2$
 preconditioners also significantly reduce the spread of the eigenvalues compared with the unpreconditioned case. Overall, the clustering of the preconditioned eigenvalues around $1$ suggests improved conditioning and explains the faster convergence expected for the corresponding iterative methods.

The results in Table~\ref{tab:cond_numbers} show that the original matrix $\mathcal K$ is highly ill-conditioned, with its condition number increasing significantly as the problem size grows. In contrast, the preconditioned matrices $\mathscr{P}_{\text{BRS}}^{-1}\mathcal K$ and $\mathscr{P}_{\text{RBRS}}^{-1}\mathcal K$ exhibit condition numbers close to one for all test problems, indicating a substantial improvement in the conditioning of the system. This demonstrates the effectiveness and robustness of the iterative solvers for both preconditioners, even in high-dimensional problems.
\end{example}
\section{Conclusions}\label{SEC6}
This paper proposed two preconditioners, namely the BRS and RBRS preconditioners, to solve the ILS problem using a regularized matrix splitting of the coefficient matrix $\mathcal{K}$ of the augmented system. Using the same splitting, we also proposed the BRS iterative scheme to solve the augmented system corresponding to the ILS problem. Convergence analysis has been done for the iterative scheme, and it has been shown that the method converges unconditionally for any $\alpha>0.$ Furthermore, spectral properties analysis of the preconditioned matrices has been performed for both preconditioners. It was shown that the eigenvalues of the BRS-preconditioned matrix lie strictly inside a circle in the complex plane with a spectral radius strictly less than one. For the RBRS preconditioned matrix, it was proved that it possesses \(p\) eigenvalues equal to \(1\), while all remaining eigenvalues are real and lie in \((0, 1]\). These favorable spectral properties provide a theoretical explanation for the effectiveness of the proposed preconditioners. Finally, comprehensive numerical experiments on different classes of test problems demonstrate the robustness and superiority of the proposed BRS and RBRS preconditioners over several existing block splitting preconditioners. 
\medskip


\section*{Declarations}
\subsection*{Funding} 
Not applicable.
\subsection*{Competing interests}
The authors have no competing interests.
\subsection*{Ethics approval} 
Not applicable.
\subsection*{Data availability} 
Not applicable.
\subsection*{Materials availability} 
Not applicable.

\bibliography{sn-bibliography}

\end{document}